\documentclass[11pt]{amsart}%
\usepackage{amssymb}
\usepackage{amsmath}
\usepackage{amsthm}
\usepackage{amscd,bm}
\usepackage{enumitem,graphicx,tikz, pst-node, pst-plot, pstricks}
\usepackage[margin=1.0in]{geometry}%
\usepackage{hyperref}
\usepackage{xcolor}
\usepackage{cite,bbm}

\DeclareMathOperator{\supp}{supp}
\DeclareMathOperator{\loc}{loc}

\DeclareMathOperator{\Ext}{Ext}

\theoremstyle{plain}
\newtheorem{theorem}{Theorem}[section]
\newtheorem{proposition}{Proposition}[section]
\newtheorem{corollary}{Corollary}[section]
\newtheorem{lemma}{Lemma}[section]

\newtheorem{definition}{Definition}[section]

\theoremstyle{remark}
\newtheorem{remark}{Remark}[section]

\newtheorem{assumption}{Assumption}[section]

\begin{document}

\title{Boundary value problems on non-Lipschitz uniform domains: Stability, compactness and the existence of optimal shapes}
\author{Michael Hinz$^1$}
\address{$^1$ Fakult\"{a}t f\"{u}r Mathematik, Universit\"{a}t Bielefeld, Postfach 100131, 33501 Bielefeld,
Germany}
\thanks{$^1$ Research supported in part by the DFG IRTG 2235: 'Searching for the regular in the irregular: Analysis of singular and random systems'. The project was finished during a stay at CentraleSup\'elec, Univ\'ersit\'e Paris-Saclay, whose kind hospitality is gratefully acknowledged.}
\email{mhinz@math.uni-bielefeld.de}
\author{Anna Rozanova-Pierrat$^2$}
\address{$^2$ Laboratory  Math\'ematiques et Informatique pour la Complexit\'e et les Syst\`emes, CentraleSup\'elec, Univ\'ersit\'e Paris-Saclay, Campus de Gif-sur-Yvette, Plateau de Moulon, 
3 rue Joliot Curie, 91190 Gif-sur-Yvette, France}
\email{anna.rozanova-pierrat@centralesupelec.fr}
\author{Alexander Teplyaev$^3$}
\address{$^3$Department of Mathematics
University of Connecticut
Storrs, CT 06269-1009
USA}
\email{alexander.teplyaev@uconn.edu}

\begin{abstract}
We study boundary value problems for bounded uniform domains in $\mathbb{R}^n$, $n\geq 2$, with non-Lipschitz (and possibly fractal) boundaries. We prove Poincar\'e inequalities with trace terms and uniform constants for uniform $(\varepsilon,\infty)$-domains within bounded common confinements. We then introduce generalized Dirichlet, Robin and Neumann problems for Poisson type equations and prove the Mosco convergence of the associated energy functionals along sequences of suitably converging domains. This implies a stability result for weak solutions, and this also implies the norm convergence of the associated resolvents and the convergence of the corresponding eigenvalues and eigenfunctions. Based on our earlier work, we prove compactness results 
for parametrized classes of admissible domains, energy functionals and weak solutions. Using these results, we can verify the existence of optimal shapes in these classes.

\tableofcontents
\end{abstract}

\maketitle

\section{Introduction}

This article deals with boundary value problems for elliptic equations that involve rough shapes. To focus on the main issues, we write our results for Poisson type equations; a generalization to other elliptic equations is straightforward. We consider  domains $\Omega\subset \mathbb{R}^n$, $n\geq 2$, not necessarily Lipschitz, together with closed sets $\Gamma\subset \overline{\Omega}$ that support a nonzero finite Borel measure $\mu$ of a certain upper regularity. The set $\Gamma$ is treated as a generalized boundary on which boundary conditions are imposed. Boundary value problems in the more narrow sense that $\Gamma=\partial\Omega$ are included in this setup as a special case. We study the behavior of weak, i.e. variational, solutions for varying domains $\Omega$ and measures $\mu$, and we address stability, compactness and the existence of optimal (energy minimizing) shapes. We are mainly interested in the case of the Robin problem, for which the literature on the existence of optimal shapes in admissible classes of rough domains is still sparse.
The problem is quite nontrivial, because it involves integrals (on $\Gamma$, which could be  $\Gamma=\partial\Omega$) with respect to a measure $\mu$ which may be a non-smooth and non-doubling. The case of the Dirichlet problem appears as a particular case, so we do not exclude it, although it can be dealt with using much easier known methods, \cite{BB2005,HENROT-e}.

	This article is the central part of a five paper series. In the foundational papers   
	\cite{HMR-PRT2021,HINZ-2020,DRPT-2020} we develop general techniques how to deal with boundary value problems in  non-Lipschitz uniform domains whose boundaries are not $d$-sets but satisfy only weaker upper and lower Ahlfors regularity estimates with possibly different lower and upper dimensions and no doubling property. 
	In 
	\cite{DHRPT-2021}
	we apply these results to the study of nonlinear Westervelt equation modeling the optimal absorption of ultrasound waves by a reflective
	or isolating boundary in   non-Lipschitz uniform domains.

In this paper show \emph{four main results}. The \emph{first main result} is about \emph{families of Poincar\'e inequalities with uniform constants}, Theorem \ref{T:uniPoincare}. It states that if $n\geq 2$, $D\subset \mathbb{R}^n$ is a bounded Lipschitz confinement, $\varepsilon>0$, $n-p<d\leq n$ and $c_d>0$, then for any $(\varepsilon,\infty)$-domain $\Omega\subset D$ and any nonzero Borel measure $\mu$ supported in $\overline{\Omega}$ and upper $d$-regular with constant $c_d$, (\ref{E:upperdreg}), a Poincar\'e inequality for $W^{1,p}(\Omega)$-function holds with a uniform constant not depending on the choice of $\Omega$ or $\mu$. The stated condition on $d$ ensures the existence of a well-defined trace on $\Gamma$ in the sense of \cite[Theorem 6.2.1 and Section 7]{AH96}, see Corollary \ref{C:tracesimple}. For $p=2$ we then   formulate abstract boundary value problems of Dirichlet, Robin and Neumann type for the linear Poisson equation on $W^{1,2}$-extension domains. We next consider the situation that $(\Omega_m)_m\subset D$ is a sequence of $(\varepsilon,\infty)$-domains $\Omega_m$ converging to some domain $\Omega$ in the Hausdorff sense and in the sense of characteristic functions and that $(\mu_m)_m$ is a sequence of measures $\mu_m$ supported inside $\overline{\Omega}_m$, respectively, all upper $d$-regular with the same constant $c_d$, and weakly converging to some measure $\mu$. Theorem \ref{T:Mosco} then tells that the energy functionals associated with the respective type of boundary value problem converge in the sense of Mosco. 

Together with Theorem \ref{T:uniPoincare} this leads to our \emph{second main result}, Theorem \ref{T:compactnesspropvar}. It is a \emph{stability result} and states that in the homogeneous case the (continuations by zero of the) weak solutions converge in $L^2(D)$ to the (continuations by zero of the) weak solution to a problem of the same type on the limit domain. We also prove norm resolvent convergence and the convergence of the corresponding spectral projectors, eigenvalues and eigenfunctions, Theorem \ref{T:normresconv}. To proceed we then introduce parametrized classes $U_{ad}(D,D_0,\varepsilon, s, \bar c_{s}, \bar c, d, c_d)$ of what we call \emph{shape admissible triples} $(\Omega,\nu,\mu)$, Definition \ref{DefShapeAdmis}. This definition is a generalization of \cite[Definition 2]{HINZ-2020}. The parameters $D$, $\varepsilon$, $d$ and $c_d$ are as above. Any $(\Omega,\nu,\mu)$ in the class consists of an $(\varepsilon,\infty)$-domain $\Omega$ contained in $D$ and containing a prescribed fixed open set $D_0$, and two measures $\mu$ and $\nu$. As before, the measure $\mu$ is supported in $\overline{\Omega}$ and upper $d$-regular, (\ref{E:upperdreg}), with constant $c_d$. The second measure $\nu$ has support $\partial\Omega$, has total mass bounded above by $\bar{c}>0$ and is lower $s$-regular in the sense  of (\ref{E:lowreg-w}), with fixed exponent $n-1\leq s< n$ and constant $\bar{c}_s>0$. These measures $\nu$ are introduced as topological tool, Remark \ref{R:boundaryvol}. Within this setup an improvement of a compactness result from \cite[Theorem 3]{HINZ-2020} holds, Theorem \ref{T:compact}. It tells that given such a parametrized class of shape admissible triples, any sequence $((\Omega_m,\nu_m,\mu_m)_m)$ it contains must have a subsequence that converges to a shape admissible triple in the same class in the Hausdorff sense, the sense of characteristic functions, the sense of compacts and the sense of weak convergence of the measures $\nu_m$ and the measures $\mu_m$. This enables our \emph{third main result}, Theorem \ref{T:Moscocompact}, which tells that the corresponding energy forms inherit this \emph{compactness} and entails that weak solutions to the respective problems have convergent subsequences as well, Corollary \ref{C:compactness}. This fact then provides easy access to our \emph{fourth main result}, namely the \emph{existence of optimal shapes} within a given parametrized class, formulated in Theorems \ref{T:opti} and \ref{T:optimu}.

These results are key contributions to a larger research program on shape optimization involving rough shapes with boundaries that may be fractal. In the long run a theory of shape optimization robust enough to discuss rough shapes with fractal boundaries could have a significant impact upon fields like acoustical engineering, biomimetic architecture and surface design. In \cite{MAGOULES-2020} mixed boundary value problems for the Helmholtz equation had been considered. It had been observed that since integrals over the boundary are involved, classes of Lipschitz admissible shapes may not contain the shape realizing the infimum energy over the class, see also \cite[Section 5]{HMR-PRT2021}. In \cite[Definition 2]{HINZ-2020} we therefore proposed larger classes of admissible shapes involving bounded $(\varepsilon,\infty)$-domains and proved that in these classes optimal (i.e. energy minimizing) shapes exist, \cite[Section 7]{HINZ-2020}. Note that also \cite[Theorem 3.2]{BucurGiacomini2016} used the idea to enlarge the classes of admissible shapes beyond Lipschitz, although in a methodologically different context. The equations discussed here and in \cite{HINZ-2020, HMR-PRT2021, MAGOULES-2020} are linear. In \cite{DRPT-2020} the authors proved well-posedness results for damped linear wave equations and for the nonlinear Westervelt equation, they also established approximation theorems in specific situations. The present article now opens the way to prove the existence of optimal shapes for the nonlinear Westervelt equation within parametrized classes of bounded uniform domains with fractal boundaries. This is done in \cite{DHRPT-2021}, and the proof there relies on Theorems \ref{T:uniPoincare} and \ref{T:compact}. Theorem \ref{T:uniPoincare} also marks a slight difference between \cite{DRPT-2020,HINZ-2020} and the present article: There uniform Poincar\'e inequalities are obtained using Dirichlet conditions on some part of the boundary, here Dirichlet conditions are not necessary, we can instead rely on Theorem \ref{T:uniPoincare}.

There is a huge amount of literature on Poincar\'e inequalities, we mention only a few sources related to our text. A very general presentation can be found in \cite[Section 7]{Egert2015}. Poincar\'e inequalities with uniform constants for classes of bounded and uniformly Lipschitz domains were discussed in \cite{BoulkhemairChakib2007}, their Theorem 2 is a Lipschitz predecessor of our Theorem \ref{T:uniPoincare}. Poincar\'e inequalities with uniform constants for domains satisfying certain cone conditions were proved in \cite{Ruiz2012} and \cite{Thomas2015}. Optimal constants (in the sense of isoperimetry) for Poincar\'e inequalities involving trace terms on bounded Lipschitz domains were determined in \cite{BucurGiacominiTrebeschi2019}. 

Robin boundary value problems on arbitrary domains were studied in \cite{DANERS-2000}, based on former ideas of Maz'ya, \cite[Section 4.11]{MAZ'JA-1985}. Within this approach traces on the boundary are defined in a somewhat abstract manner and always considered with respect to the $(n-1)$-dimensional Hausdorff measure. The traces can 'disconnect' from the function on the domain, and Robin conditions can reduce to Dirichlet conditions, \cite[Remarks 3.5 (a) and (d)]{DANERS-2000}. Based on this approach, Robin problems for linear and nonlinear equations on varying domains were investigated in \cite{DancerDaners1997}, where it was shown that in the limit problem a different type of boundary condition can appear. Within our 'parametrized' setup such effects are excluded. An abstract potential theoretic approach to Robin problems on arbitrary domains was introduced in \cite{ARENDT-WARMA2003} (with H\"older regularity addressed in \cite{NYSTROM-1994,NYSTROM-1996,VS15}), and as far as only well-posedness is concerned, it contains the basic setup of the present paper as a quantitative special case (up to some technicalities). Using variational methods, \cite{BB2005, Braides,DALMASO-1993}, Robin problems on sequences of confined Lipschitz domains were studied in the article \cite{BucurGiacominiTrebeschi2016}, which partially inspired the present results. Under the assumption that the domains converge in the Hausdorff sense, their total volumes converge and the measures on the boundary are diffuse enough and bounded below by a (uniform) constant times the $(n-1)$-dimensional Hausdorff measure, the weak solutions of the individual problems were shown to converge to the weak solution of the limit problem, which is again Robin, \cite[Theorems 3.1 and 3.4]{BucurGiacominiTrebeschi2016}. Also norm resolvent convergence and the convergence of the corresponding eigenvalues were proved,  \cite[Theorem 3.2 and Corollary 3.5]{BucurGiacominiTrebeschi2016}. 
No uniform (in this case 'equi-Lipschitz') condition was imposed on the geometry of the domains. This is different in the present article, where the fixed parameter $\varepsilon$ forces a uniform 'worst case' control, although within the more general category of bounded uniform domains. The stability results in \cite{BucurGiacominiTrebeschi2016} were established using the fact that the indicators of Lipschitz domains with finite measure on the boundary are $BV$-functions; extension operators are not needed. Their proof also uses a Faber-Krahn inequality for Robin problems established in \cite{Daners2006}; nonlinear generalizations of this inequality were provided in \cite{BucurGiacomini2015} via regularity results for free discontinuity problems. A free discontinuity approach to shape optimization for Robin problems on varying domains was proposed in \cite{BucurGiacomini2016}. The original problem they addressed involves a very general nonlinear energy integral, from which the energy functional to be minimized is obtained by taking an infimum of $W^{1,p}$-functions on the respective domain, subject to a pinning ('obstacle type') condition on a given fixed open subset inside the domain. As already mentioned above, they, too, relaxed the original problem. For the relaxed model the original energy integral was replaced by a free discontinuity functional (involving jumps and traces in the sense of $BV$-functions), and the class of admissible shapes was enlarged to a collection of open domains $\Omega$ with countably rectifiable boundaries $\partial \Omega$ of finite $(n-1)$-dimensional Hausdorff measure. The existence of optimal shapes for the relaxed model was then proved in \cite[Theorem 3.2]{BucurGiacomini2016}. In the present article we do not insist on any rectifiability properties, the boundaries of optimal shapes may have any Hausdorff dimension in $[n-1,n)$, see Remark \ref{R:boundarycase}.

In Section \ref{S:traces} we collect preliminary material about traces of Sobolev functions on closed subsets. In Section \ref{S:Poincare} we prove the Poincar\'e inequalities with uniform constants. Section \ref{S:bvp} discusses generalized boundary value problems, energy functionals and a first compactness lemma. The Mosco convergence of energy functionals along varying domains is proved in Section \ref{S:Mosco}, also the stability result for weak solutions is shown there. Norm resolvent convergence and spectral convergence are stated and proved in Section \ref{S:eigenvalues}. In Section \ref{S:compactness} we define the parametrized classes of admissible triples and state the compactness results. The existence of optimal shapes is shown in Section \ref{S:opti}.

\section{Traces on closed subsets}\label{S:traces}

We collect preliminary facts on traces of Sobolev functions to closed subsets of $\mathbb{R}^n$. For any $1<p<+\infty$ and $\beta>0$ let $H^{\beta,p}(\mathbb{R}^n)$ denote the $L^p$-Bessel-potential space of order $\beta$, \cite{AH96, Triebel78, TRIEBEL-1997}, that is, the space of all $f\in \mathcal{S}'(\mathbb{R}^n)$ such that $((1+|\xi|^2)^{\beta/2}\hat{f})^\vee \in L^p(\mathbb{R}^n)$. Here $f\mapsto \hat{f}$ denotes the Fourier transform and $f\mapsto \check{f}$ its inverse, and $\mathcal{S}'(\mathbb{R}^n)$ is the space of tempered distributions. Endowed with the norm 
\[\|f\|_{H^{\beta,p}(\mathbb{R}^n)}:=\|((1+|\xi|^2)^{\beta/2}\hat{f})^\vee\|_{L^p(\mathbb{R}^n)}\] 
the space $H^{\beta,p}(\mathbb{R}^n)$ is Banach, and Hilbert if $p=2$.

We write $\lambda^n$ to denote the $n$-dimensional Lebesgue measure and $B(x,r)$ for the open ball in $\mathbb{R}^n$ of radius $r>0$ centered at $x$. Given $f\in L^1_{\loc}(\mathbb{R}^n)$, the limit 
\begin{equation}\label{E:redefinition}
\widetilde{f}(x)=\lim_{r\to 0}\frac{1}{\lambda^n(B(x,r))}\int_{B(x,r)}f(y)dy
\end{equation}
exists at $\lambda^n$-a.e. $x\in \mathbb{R}^n$. Suppose that $1<p<+\infty$ and $f\in H^{\beta,p}(\mathbb{R}^n)$. If $\beta >n/p$, then the function $f$ is H\"older continuous, \cite[2.8.1]{Triebel78}, the limit \eqref{E:redefinition} exists at any point $x$ and equals $f(x)$. If $0<\beta\leq n/p$, then it exists at $H^{\beta,p}(\mathbb{R}^n)$-quasi every $x\in\mathbb{R}^n$ and $\widetilde{f}$ defines a $H^{\beta,p}(\mathbb{R}^n)$-quasi continuous version of $f$, \cite[Theorem 6.2.1]{AH96}. 

Now let $\mu$ be a Borel measure with support $\Gamma=\supp\mu$ and satisfying the (local) upper regularity condition
\begin{equation}\label{E:upperdreg}
\mu(B(x,r))\leq c_dr^d, \quad x\in\Gamma,\quad 0<r\leq 1,
\end{equation}
with some $0< d\leq n$. This implies that the Hausdorff dimension of $\Gamma$ is at least $d$, \cite{FALCONER,MATTILA}, it also implies that $\mu$ is a locally finite measure. If condition (\ref{E:upperdreg}) holds and $\beta$ is large enough, then we can conclude that $\mu$ charges no set of zero $H^{\beta,p}(\mathbb{R}^n)$-capacity and consequently the limit in \eqref{E:redefinition} exists for $\mu$-a.e. $x\in \Gamma$. The following is a version of \cite[Theorem 4]{HINZ-2020}. We use the convention $1/0:=+\infty$, and we write $L^0(\Gamma,\mu)$ for the space of $\mu$-equivalence classes of Borel functions on $\Gamma$.

\begin{theorem}\label{ThTracecheap}
Let  $0< d\leq n$, $(n-d)/p<\beta$ and suppose that $\mu$ is a Borel measure with support $\Gamma=\supp\mu$ and satisfying \eqref{E:upperdreg}. Then 
\begin{equation}\label{E:traceopdef}
\operatorname{Tr}_\Gamma f:= \widetilde{f}
\end{equation}
defines a linear map $\operatorname{Tr}_\Gamma$ from $H^{\beta,p}(\mathbb{R}^n)$ into $L^0(\Gamma,\mu)$. If in addition $\beta\leq  n/p$  and $p<q<pd/(n-p\beta)$, then $\operatorname{Tr}_\Gamma$ is a bounded linear operator from $H^{\beta,p}(\mathbb{R}^n)$ into $L^q(\Gamma,\mu)$, and we have 
\[\left\|\operatorname{Tr}_\Gamma f\right\|_{L^q(\Gamma,\mu)}\leq c_{\operatorname{Tr}}\left\|f\right\|_{H^{\beta,p}(\mathbb{R}^n)},\quad  f\in H^{\beta,p}(\mathbb{R}^n),\] 
with a constant $c_{\mathrm{Tr}}>0$ depending only on $n$, $p$, $\beta$, $d$, $c_d$ and $q$. If $\Gamma$ is compact, then $\operatorname{Tr}_\Gamma:H^{\beta,p}(\mathbb{R}^n)\to L^q(\Gamma,\mu)$ is a compact operator.
\end{theorem}

\begin{proof}
The theorem follows from \cite[Theorems 7.2.2 and 7.3.2 and Propositions 5.1.2, 5.1.3 and 5.1.4]{AH96}.
\end{proof}

We write $\mathrm{Tr}_{\Gamma}(H^{\beta,p}(\mathbb{R}^n))$ for the image of $H^{\beta,p}(\mathbb{R}^n)$ under $\mathrm{Tr}_{\Gamma}$ in $L^0(\Gamma,\mu)$. The following is immediate from general theory.

\begin{corollary}\label{C:tracespace}
Let  $0< d\leq n$ and $(n-d)/p<\beta$. Suppose that $\mu$ is a Borel measure with support $\Gamma=\supp\mu$ and satisfying \eqref{E:upperdreg}. The map 
\begin{equation}\label{E:quotientnorm}
\varphi \mapsto \left\|\varphi\right\|_{\mathrm{Tr}_\Gamma(H^{\beta,p}(\mathbb{R}^n))}:=\inf\left\lbrace \left\|g\right\|_{H^{\beta,p}(\mathbb{R}^n)}:\  g\in H^{\beta,p}(\mathbb{R}^n), \mathrm{Tr}_\Gamma g=\varphi\ \text{$\mu$-a.e.}\right\rbrace 
\end{equation}
is a norm that makes $\mathrm{Tr}_\Gamma(H^{\beta,p}(\mathbb{R}^n))$ a Banach space, and for $p=2$ Hilbert. 
\end{corollary}

We briefly point out that the trace space $\mathrm{Tr}_\Gamma(H^{\beta,p}(\mathbb{R}^n))$ does not depend on the choice of an individual measure $\mu$ but only on the choice of an equivalence class of measures. Recall that if $\mu$ is a Borel measure charging no set of zero $H^{\beta,p}(\mathbb{R}^n)$-capacity, then a set $\Gamma_q\subset \mathbb{R}^n$ is called an \emph{$H^{\beta,p}(\mathbb{R}^n)$-quasi support} of $\mu$ if (i) we have $\mu(\mathbb{R}^n\setminus \Gamma_q)=0$, (ii) the set $\Gamma_q$ is $H^{\beta,p}(\mathbb{R}^n)$-quasi closed and (iii) if $\widetilde{\Gamma}_q$ is another set with these properties, then $\Gamma_q\subset \widetilde{\Gamma}_q$ $H^{\beta,p}(\mathbb{R}^n)$-quasi everywhere. See \cite[Section 2.9]{Fuglede71}. Two $H^{\beta,p}(\mathbb{R}^n)$-quasi supports of $\mu$ can differ only by a set of zero 
$H^{\beta,p}(\mathbb{R}^n)$-capacity. Since any closed set is $H^{\beta,p}(\mathbb{R}^n)$-quasi closed, we may always assume that 
an $H^{\beta,p}(\mathbb{R}^n)$-quasi support $\Gamma_q$ of $\mu$ satisfies $\Gamma_q\subset \Gamma=\supp\mu$. See \cite[Section 4.6]{FOT94} for the case $p=2$.

\begin{proposition}\label{P:tracespacedoesnotdepend}
Let  $0< d\leq n$ and $(n-d)/p<\beta$. Suppose that $\mu$ is a Borel measure with support $\Gamma=\supp\mu$ and satisfying \eqref{E:upperdreg}. If $\Gamma_q$ is an $H^{\beta,p}(\mathbb{R}^n)$-quasi support of $\mu$, then 
\[\left\|\varphi\right\|_{\mathrm{Tr}_\Gamma(H^{\beta,p}(\mathbb{R}^n))}=\inf\left\lbrace \left\|g\right\|_{H^{\beta,p}(\mathbb{R}^n)}:\  g\in H^{\beta,p}(\mathbb{R}^n), \ \widetilde{g}=\varphi\ \text{$H^{\beta,p}(\mathbb{R}^n)$-q.e. on $\Gamma_q$}\right\rbrace,\ \varphi\in H^{\beta,p}(\mathbb{R}^n). \]
If $\mu'$ is another Borel measure satisfying \eqref{E:upperdreg} (with a possibly different constant $c_d'>0$) and having $\Gamma_q$ as a quasi support, then the trace space $(\mathrm{Tr}_\Gamma(H^{\beta,p}(\mathbb{R}^n)),\left\|\cdot\right\|_{\mathrm{Tr}_\Gamma(H^{\beta,p}(\mathbb{R}^n))})$ defined using $\mu'$ is the same as the one defined using $\mu$.
\end{proposition}

\begin{remark}\label{R:tracespacedoesnotdepend}
Suppose that if $\mu$ and $\mu'$ are two Borel measures equivalent to each other and both satisfying \eqref{E:upperdreg} (possibly with different multiplicative constants). Then any $H^{\beta,p}(\mathbb{R}^n)$-quasi support of $\mu$ is also one for $\mu'$ and vice versa.
\end{remark}

\begin{proof}
The first statement one can follow the corresponding arguments in the proof of \cite[Theorem 4.6.2]{FOT94}, they remain valid for $p\neq 2$. The second is an immediate consequence.
\end{proof}

In the case of finite measures we can conclude further properties.

\begin{corollary}\label{C:tracesimple}
Let  $0< d\leq n$ and $(n-d)/p<\beta$. Suppose that $\mu$ is a finite Borel measure with support $\Gamma=\supp\mu$ and satisfying \eqref{E:upperdreg}. Then $\operatorname{Tr}_\Gamma$ as in \eqref{E:traceopdef} defines a bounded linear operator from $H^{\beta,p}(\mathbb{R}^n)$ into $L^p(\Gamma,\mu)$ with operator norm depending only on $n$, $p$, $\beta$, $d$, $c_d$ and $\mu(\Gamma)$. This operator is compact if $\Gamma$ is. Seen as an operator from $H^{\beta,p}(\mathbb{R}^n)$ onto $\mathrm{Tr}_\Gamma(H^{\beta,p}(\mathbb{R}^n))$ normed by (\ref{E:quotientnorm}), the operator $\operatorname{Tr}_\Gamma$ is a linear contraction.
The space $\mathrm{Tr}_{\Gamma}(H^{\beta,p}(\mathbb{R}^n))$ is dense in $L^p(\Gamma,\mu)$.
\end{corollary}

\begin{proof}
For $\beta>n/p$ the pointwise redefinition (\ref{E:redefinition}) agrees with the original function, and by the Sobolev embedding
of $H^{\beta,p}(\mathbb{R}^n)$ into a space of bounded H\"older continuous functions the operator $\operatorname{Tr}_\Gamma$ is bounded from $H^{\beta,p}(\mathbb{R}^n)$ into $L^p(\Gamma,\mu)$ because $\mu$ is finite. If $\Gamma$ is compact, $\operatorname{Tr}_\Gamma$ is compact by Arzela-Ascoli. If $(n-d)/p<\beta\leq  n/p$ we can find some $q>p$ such that Theorem \ref{ThTracecheap} applies, and since $\mu$ is finite, the embedding of $L^q(\Gamma,\mu)$ into $L^p(\Gamma,\mu)$ is bounded. The contractivity is clear from (\ref{E:quotientnorm}). The claimed density follows from the fact that $C_c^\infty(\mathbb{R}^n)$ is dense in $W^{1,2}(\mathbb{R}^n)$ together with the Stone-Weierstrass theorem.
\end{proof}

\begin{remark}
It is well-known that under stronger assumptions on $\mu$ one can identify the trace space $\mathrm{Tr}_{\Gamma}(H^{\beta,p}(\mathbb{R}^n))$ explicitly as a Besov space, \cite{JONSSON-1994, JONSSON-2009,JONSSON-1984}, see also \cite[Section 5.1]{HINZ-2020}. 
\end{remark}

For later use we fix a simple observation on the consistency of trace procedures.
\begin{corollary}
Let  $0< d\leq n$ and $(n-d)/p<\beta$. Suppose that $\mu$ is a finite Borel measure with support $\Gamma=\supp\mu$ and satisfying \eqref{E:upperdreg}. If $\Gamma'$ is a closed subset of $\Gamma$ with $\mu(\Gamma')>0$ then $\operatorname{Tr}_{\Gamma'}$ is a bounded linear operator from $H^{\beta,p}(\mathbb{R}^n)$ into $L^p(\Gamma',\mu)$, and for any $f\in H^{\beta,p}(\mathbb{R}^n)$ we have $\operatorname{Tr}_{\Gamma'}f=\operatorname{Tr}_{\Gamma}f|_{\Gamma'}$ $\mu$-a.e. on $\Gamma'$.
\end{corollary} 
\begin{proof}
Since the closed set $\Gamma'$ is the support of the measure $\mu':=\mu|_{\Gamma'}$, which satisfies (\ref{E:upperdreg}) at all $x\in \Gamma'$, the claims follow from Corollary \ref{C:tracesimple} and (\ref{E:traceopdef}).
\end{proof}

Now let $\Omega\subset \mathbb{R}^n$ be a Lipschitz domain. For $1<p<+\infty$ and $\beta>0$ we write 
\[H^{\beta,p}(\Omega)=\{f\in D'(\Omega): \text{$f=g|_\Omega$ for some $g\in H^{\beta,p}(\mathbb{R}^n)$}\},\] 
where $D'(\Omega)$ denotes the space of Schwartz distributions on $\Omega$. Endowed with the norm  
\[\left\|f\right\|_{H^{\beta,p}(\Omega)}:=\inf \{ \left\|g\right\|_{H^{\beta,p}(\mathbb{R}^n)}:\ g\in H^{\beta,p}(\mathbb{R}^n), f=g|_\Omega\ \text{in $D'(\Omega)$}\}.\] 
it becomes a Banach space (and Hilbert for $p=2$). See for instance \cite[Section 2.2]{Triebel2002}. 

Suppose that $\mu$ is a finite Borel measure as specified in Corollary \ref{C:tracesimple} and that $\Gamma=\supp\:\mu\subset \overline{\Omega}$. Given $f\in H^{\beta,p}(\Omega)$, we set 
\[\mathrm{Tr}_{\Omega,\Gamma}f:=\mathrm{Tr}_{\Gamma}g,\]
where $g$ is an element of $H^{\beta,p}(\mathbb{R}^n)$ such that $f=g|_\Omega$ in $D'(\Omega)$. By the arguments used to prove \cite[Theorem 6.1]{Biegert2009} one can see that $\mathrm{Tr}_{\Omega,\Gamma}$ is a well-defined linear map
$\mathrm{Tr}_{\Omega,\Gamma}:H^{\beta,p}(\Omega)\to L^p(\Gamma,\mu)$ in the sense that if $g,h\in H^{\beta,p}(\mathbb{R}^n)$ satisfy $g=h$ $\lambda^n$-a.e. in $\Omega$, then $\mathrm{Tr}_{\Gamma}g=\mathrm{Tr}_{\Gamma}h$ $\mu$-a.e. on $\Gamma$. It is a bounded linear operator, and the image $\mathrm{Tr}_{\Omega,\Gamma}(H^{\beta,p}(\Omega))$ agrees with $\mathrm{Tr}_\Gamma(H^{\beta,p}(\mathbb{R}^n))$.

Later we will make use of the following variant of \cite[Theorem 5]{HINZ-2020} on the convergence of integrals of traces in a fixed confinement $D$. In contrast to the result there we now only assume weak convergence of functions.

\begin{lemma}\label{L:traceconvergence}  
Let $D\subset\mathbb{R}^n$ be a bounded Lipschitz domain, $0< d\leq n$ and $(n-d)/p<\beta$. Let $(\mu_m)_m$ be a sequence of Borel measures with supports $\Gamma_m=\supp\mu_m$ contained in $\overline{D}$ and such that \eqref{E:upperdreg} holds for all $m$ with the same constant $c_d$. Suppose that $(\mu_m)_m$ converges weakly to a Borel measure $\mu$. Then $\Gamma:=\supp\mu$ is contained in $\overline{D}$, and if $(v_m)_m\subset H^{\beta,p}(D)$ converges to some $v$ weakly in $H^{\beta,p}(D)$, then 
	\[\lim_{m\to\infty}\int_{\Gamma_{m}} \mathrm{Tr}_{D,\Gamma_m}v_{m}\:d\mu_{m} =\int_{\Gamma} \mathrm{Tr}_{D,\Gamma} v\:d\mu\]
	and
	\[\lim_{m\to\infty}\int_{\Gamma_{m}} |\mathrm{Tr}_{D,\Gamma_m}v_{m}|^p\:d\mu_{m} =\int_{\Gamma} |\mathrm{Tr}_{D,\Gamma} v|^p \:d\mu.\]
\end{lemma}

\begin{proof} Note first that since all $\mu_m$ satisfy (\ref{E:upperdreg}) and are supported in the compact set $\overline{D}$, we have $\sup_m\mu_m(\mathbb{R}^n)<+\infty$. The fact that $\Gamma\subset \overline{D}$ is clear from weak convergence. To see the stated limit relations, let $(\varphi_j)_j \subset C_c^\infty(\overline{D})$ be a sequence converging to $v$ in $H^{\beta,p}(D)$. Similarly as in the proof of \cite[Theorem 5.2]{CAPITANELLI-2010}, we observe that
	\begin{multline}
	\left|\int_{\Gamma_m} \mathrm{Tr}_{D,\Gamma_m} v_{m} \:d\mu_{m} -\int_{\Gamma} \mathrm{Tr}_{D,\Gamma} v  \:d\mu \right|\\
	\le \left|\int_{\Gamma_{m}} \mathrm{Tr}_{D,\Gamma_{m}} v_{m}  \:d\mu_{m} -\int_{\Gamma_{m}} \mathrm{Tr}_{D,\Gamma_{m}} v  \:d\mu_{m}\right|
	+\left|\int_{\Gamma_{m}} \mathrm{Tr}_{D,\Gamma_{m}} v \: d\mu_{m} -\int_{\Gamma_{m}} \varphi_{j} \: d\mu_{m}\right| \\
	+\left|\int_{\Gamma_{m}} \varphi_{j}\: d\mu_{m} -\int_{\Gamma} \varphi_{j} \: d\mu\right|
	+\left|\int_{\Gamma} \varphi_{j} \:d\mu -\int_{\Gamma}\mathrm{Tr}_{\Gamma} v \: d\mu\right|.\label{EqBigEstInt-}
	\end{multline}
	Now let $(n-d)/p<\beta'<\beta$. By H\"older's inequality and Corollary \ref{C:tracesimple} the first term on the right hand side of \eqref{EqBigEstInt-} is bounded by  
	\begin{align}
	\int_{\Gamma_m} |\mathrm{Tr}_{D,\Gamma_m}(v_m-v)|\:d\mu_m &\leq \mu_m(\mathbb{R}^n)^{1/q}\left\|\mathrm{Tr}_{D,\Gamma_m}(v_m-v)\right\|_{L^p(\Gamma_m,\mu_m)}\notag\\
	&\leq c_{\mathrm{Tr}}\:\mu_m(\mathbb{R}^n)^{1/q}\left\|v_m-v\right\|_{H^{\beta',p}(D)}\notag
	\end{align}
	with $c_{\mathrm{Tr}}>0$ independent of $m$. Since the embedding $H^{\beta,p}(D) \subset H^{\beta',p}(D)$ is compact, \cite[Theorem 2.7]{Triebel2002}, this goes to zero as $m\to \infty$. When $j\to \infty$ the second term and the last in~\eqref{EqBigEstInt-} converge to zero uniformly in $m$.
	The third term converges to zero as $m\to\infty$ by weak convergence of measures.
	
	The second statement follow from a slight modification of the proof of \cite[Theorem 5]{HINZ-2020}: We have an estimate analogous to (\ref{EqBigEstInt-}), but with $|\mathrm{Tr}_{D,\Gamma_m} v_m|^p$, $|\mathrm{Tr}_{D,\Gamma_m} v|^p$ and $|\varphi_j|^p$ in place of $\mathrm{Tr}_{D,\Gamma_m} v_m$, $\mathrm{Tr}_{D,\Gamma_m} v$ and $\varphi_j$. The convergence of the last three summands on the right hand side of this inequality follows from the uniform norm bound on the trace operator and the weak convergence of the measures. For the first summand we can use the mean value theorem for $s\mapsto s^p$, the reverse triangle inequality and H\"older's inequality w.r.t. $\mu_m$ to see that
\begin{align}\label{E:pvariant}
\left|\int_{\Gamma_m} \right. &\left. |\mathrm{Tr}_{D,\Gamma_m} v_m|^p\:d\mu_m-\int_{\Gamma_m}|\mathrm{Tr}_{D,\Gamma_m} v|^p\:d\mu_m\right|\\
&\leq p\:\int_{\Gamma_m}\Big||\mathrm{Tr}_{D,\Gamma_m} v_m|-|\mathrm{Tr}_{D,\Gamma_m} v_m|\Big|\Big(|\mathrm{Tr}_{D,\Gamma_m} v_m|+|\mathrm{Tr}_{D,\Gamma_m} v|\Big)^{p-1}\:d\mu_m\notag\\
&\leq p\left(\int_{\Gamma_m}|\mathrm{Tr}_{D,\Gamma_m} (v_m-v)|^p\:d\mu_m\right)^{1/p}\left(\int_{\Gamma_m}\Big(|\mathrm{Tr}_{D,\Gamma_m} v_m|+|\mathrm{Tr}_{D,\Gamma_m} v|\Big)^p\:d\mu_m\right)^{(p-1)/p}.\notag
\end{align}

Again by Corollary \ref{C:tracesimple} we have 
\[\max\left\lbrace \sup_m\left\|\mathrm{Tr}_{D,\Gamma_m} v_m\right\|_{L^p(\Gamma_m,\mu_m)}, \sup_m\left\|\mathrm{Tr}_{D,\Gamma_m} v\right\|_{L^p(\Gamma_m,\mu_m)}\right\rbrace\leq  c_{\mathrm{Tr}}'\sup_m\left\|v_m\right\|_{H^{\beta,p}(D)}\]
with $c_{\mathrm{Tr}}'>0$ independent of $m$. Using the same compact embedding as above we see that 
\[\lim_{m\to \infty} \left\|\mathrm{Tr}_{D,\Gamma_m} (v_m-v)\right\|_{L^p(\Gamma_m,\mu_m)}=0.\] 
Consequently also (\ref{E:pvariant}) converges to zero.
\end{proof}

Because the matter is systematically close to the preceding, we briefly discuss linear extension operators for the case $p=2$. This discussion will not be needed in later sections. Let $\mu$ be a nonzero Borel measure on $\mathbb{R}^n$ satisfying (\ref{E:upperdreg}) whose support $\Gamma=\supp\mu$ is a subset of $\mathbb{R}^n$. Given $\beta>0$ let $\mathring{H}^{\beta,2}(\mathbb{R}^n\setminus \Gamma)$ be the closure of $C_c^\infty(\mathbb{R}^n\setminus \Gamma)$ in the Hilbert space $H^{\beta,2}(\mathbb{R}^n)$, it coincides with the subspace of all elements $v$ such that $\widetilde{v} =0$ $H^{\beta,2}(\mathbb{R}^n)$-q.e. on $\Gamma$, \cite[Corollary 2.3.1]{FOT94}. Suppose that $0< d\leq n$, $(n-d)/2<\beta$ and that $\varphi\in \mathrm{Tr}_\Gamma(H^{\beta,2} (\mathbb{R}^n))$. For any $\alpha\geq 0$ the generalized Dirichlet problem 
\begin{equation}\label{E:Dirichletfortrace}
\begin{cases}
(\alpha-\Delta)^\beta u&=0\quad \text{on $\mathbb{R}^n\setminus \Gamma$}\\
\qquad \mathrm{Tr}_\Gamma u & =\varphi\quad \text{on $\Gamma$}
\end{cases}
\end{equation}
can be made rigorous by saying that $u\in H^{\beta,2}(\mathbb{R}^n)$ is a \emph{weak solution} to (\ref{E:Dirichletfortrace}) if
$\mathrm{Tr}_\Gamma u  =\varphi$ $H^{\beta,2} (\mathbb{R}^n)$-q.e. and 
\begin{equation}\label{E:bilinforext}
\big\langle (\alpha-\Delta)^{\beta/2}u,(\alpha-\Delta)^{\beta/2}v\big\rangle_{L^2(\mathbb{R}^n)}=0, \quad v\in \mathring{H}^{\beta,2}(\mathbb{R}^n\setminus \Gamma).
\end{equation}
The next two statements are a version of \cite[Corollary 1]{HINZ-2020}. 

\begin{proposition}\label{P:fracext}
Let  $0< d\leq n$ and $(n-d)/2<\beta$. Suppose that $\mu$ is a nonzero finite Borel measure whose support $\Gamma=\supp\mu$ is a subset of $\mathbb{R}^n$ and which satisfies \eqref{E:upperdreg}. For any $\alpha>0$ and $\varphi \in \mathrm{Tr}_\Gamma(H_2^\beta(\mathbb{R}^n))$ there is a unique weak solution $H_{\Gamma}^\alpha\varphi$  to \eqref{E:Dirichletfortrace}. 
\end{proposition}
\begin{proof}
Let $\left\langle \cdot,\cdot\right\rangle_\alpha$ denote the bilinear form on the left hand side of (\ref{E:bilinforext}). Let $w\in H^{\beta,2}(\mathbb{R}^n)$ be such that $\mathrm{Tr}_\Gamma w=\varphi$, and let $\widetilde{u}$ be the unique element of $\mathring{H}^{\beta,2}(\mathbb{R}^n\setminus \Gamma)$ such that $\left\langle \widetilde{u},v\right\rangle_\alpha=-\left\langle w,v\right\rangle_\alpha$, $v\in \mathring{H}^{\beta,2}(\mathbb{R}^n\setminus \Gamma)$. Then the function $u:=\widetilde{u}+w$ has the desired properties.
\end{proof}

We use the shortcut notation 
\begin{equation}\label{E:harmext0}
H_{\Gamma}\varphi:=H_{\Gamma}^1\varphi. 
\end{equation}
It is common to refer to $H_{\Gamma}$ as the \emph{$1$-harmonic extension operator}.

\begin{corollary}\label{C:extension}
Let  $0< d\leq n$ and $(n-d)/2<\beta$. Suppose that $\mu$ is a nonzero finite Borel measure whose support $\Gamma=\supp\mu$ is a subset of $\mathbb{R}^n$ and which satisfies \eqref{E:upperdreg}. The map 
\[H_{\Gamma}:\mathrm{Tr}_\Gamma(H^{\beta,2}(\mathbb{R}^n))\to H^{\beta,2}(\mathbb{R}^n), \quad \varphi\mapsto H_{\Gamma}\varphi,\] 
is a linear extension operator of norm one, and $\mathrm{Tr}_\Gamma(H_{\Gamma}\varphi)=\varphi$, $\varphi \in \mathrm{Tr}_\Gamma(H^{\beta,2}(\mathbb{R}^n))$. 

In the case that $0<\beta\leq 1$ any $\varphi\in \mathrm{Tr}_\Gamma(H^{\beta,2}(\mathbb{R}^n))$ has an extension $u\in H^{\beta,2}(\mathbb{R}^n)\cap L^\infty(\mathbb{R}^n)$ with $\left\|u\right\|_{H^{\beta,2}(\mathbb{R}^n)}\leq \left\|\varphi\right\|_{\mathrm{Tr}_\Gamma(H^{\beta,2}(\mathbb{R}^n))}$.
\end{corollary}

\begin{proof}
All but the last statement follow from the preceding. The last statement follows from the fact that for $0<\beta\leq 1$ the pair $(\left\|\cdot\right\|_{H^{\beta,2}(\mathbb{R}^n)}^2, H^{\beta,2}(\mathbb{R}^n))$ is a Dirichlet form, see \cite[Section I.3]{BH91} or \cite[Section 1.1]{FOT94}.
\end{proof}

\begin{remark}
The same strategy, based on a nonlinear version of (\ref{E:Dirichletfortrace}), yields extension operators also for the case $p\neq 2$, but they may not be linear. See \cite{JONSSON-2009, TRIEBEL-1997}.
\end{remark}

\section{Poincar\'e inequalities with uniform constants}\label{S:Poincare}

We turn to Poincar\'e inequalities. Let $\Omega\subset \mathbb{R}^n$ be a domain and let $1<p<+\infty$. The Sobolev space $W^{1,p}(\Omega)$ is defined as the space of all $f\in L^p(\Omega)$ for which all first order partials $\partial_j f$, defined in distributional sense, also belong to  $L^p(\Omega)$. Endowed with the norm 
\[\left\|f\right\|_{W^{1,p}(\Omega)}:=\Big(\left\|f\right\|_{L^p(\Omega)}^p+\sum_{j=1}^n\left\|\partial_j f\right\|_{L^p(\Omega)}^p\Big)^{1/p}\]
it is Banach, and for $p=2$ Hilbert. It is well-known that $W^{1,p}(\mathbb{R}^n)=H^{1,p}(\mathbb{R}^n)$ in the sense of equivalently normed vector spaces. Note that trivially, $\Omega=\mathbb{R}^n$ itself is an extension domain, and we may choose $E_{\mathbb{R}^n}$ to be the identity.

Obviously the restriction $\mathrm{Tr}_\Omega f=f|_\Omega$ of an element $f\in W^{1,p}(\mathbb{R}^n)$ to a domain $\Omega$ is an element of $W^{1,p}(\Omega)$, the operator $\mathrm{Tr}_\Omega$ is a linear contraction, and it is compact if $\Omega$ is bounded. The domain $\Omega$ is said to be a \emph{$W^{1,p}$-extension domain} if there is a bounded linear extension operator $E_\Omega: W^{1,p}(\Omega)\to W^{1,p}(\mathbb{R}^n)$. In this case we have $(Ef)|_\Omega =f$ for all $f\in W^{1,p}(\Omega)$.  Note that trivially, $\Omega=\mathbb{R}^n$ itself is an extension domain, and we may choose $E_{\mathbb{R}^n}$ to be the identity.
 
Now suppose that $\Omega$ is a $W^{1,p}$-extension domain and that $\mu$ is a nonzero finite Borel measure whose support $\supp\mu=\Gamma$ is a subset of $\overline{\Omega}$ and which satisfies \eqref{E:upperdreg} with some $0\vee (n-p)<d\leq n$. Similarly as in the last section \cite[Theorem 6.1]{Biegert2009} guarantees that 
\[\mathrm{Tr}_{\Omega, \Gamma}:=\mathrm{Tr}_{\Gamma}\circ E_{\Omega}:\ W^{1,p}(\Omega)\to L^p(\Gamma,\mu) \] 
is well defined in the sense that if $u,v\in W^{1,p}(\mathbb{R}^n)$ are such that $u=v$ $\lambda^n$-a.e. in $\Omega$, then $\mathrm{Tr}_{\Gamma}u=\mathrm{Tr}_{\Gamma}v$ $\mu$-a.e. on $\Gamma$. See also \cite[Theorem 1]{WALLIN-1991}. We have $\mathrm{Tr}_{\Omega, \Gamma}(W^{1,p}(\Omega))=\mathrm{Tr}_{\Gamma}(W^{1,p}(\mathbb{R}^n))$. Corollary \ref{C:tracesimple} implies that 
$\mathrm{Tr}_{\Omega, \Gamma}$ is a bounded linear operator, and it is compact if $\Gamma$ is. Seen as a linear operator 
$\mathrm{Tr}_{\Omega, \Gamma}:W^{1,p}(\Omega)\to \mathrm{Tr}_{\Gamma}(W^{1,p}(\mathbb{R}^n))$ it is a contraction. In the case that $\Omega=\mathbb{R}^n$ and $E_{\mathbb{R}^n}$ is the identity, we have $\mathrm{Tr}_{\mathbb{R}^n, \Gamma}=\mathrm{Tr}_{\Gamma}$.

Recall notation (\ref{E:harmext0}). We briefly record the following immediate consequence of Corollary \ref{C:extension}.

\begin{corollary}\label{C:extensiontodomain}
Let $\Omega$ be a $W^{1,2}$-extension domain and suppose that $\mu$ is a nonzero finite Borel measure satisfying (\ref{E:upperdreg}) with some $0\vee (n-2)<d\leq n$ and whose support $\supp\mu=\Gamma$ is a subset of $\overline{\Omega}$. The map 
\begin{equation}\label{E:harmext}
H_{\Gamma,\Omega}:=\mathrm{Tr}_\Omega\circ H_{\Gamma}:\: \mathrm{Tr}_{\Gamma}(W^{1,2}(\mathbb{R}^n))\to W^{1,2}(\Omega).
\end{equation}
is a linear extension operator of norm one, and $\mathrm{Tr}_{\Omega,\Gamma}(H_{\Gamma,\Omega}\varphi)=\varphi$, $\varphi \in \mathrm{Tr}_\Gamma(W^{1,2}(\mathbb{R}^n))$. 

Moreover, any $\varphi\in \mathrm{Tr}_\Gamma(W^{1,2}(\mathbb{R}^n))\cap L^\infty(\Gamma,\mu)$ has an extension $u\in W^{1,2}(\Omega)\cap L^\infty(\Omega)$ with $\left\|u\right\|_{W^{1,2}(\Omega)}\leq \left\|\varphi\right\|_{\mathrm{Tr}_\Gamma(W^{1,2}(\mathbb{R}^n))}$.
\end{corollary}

Throughout the remaining section we assume $n\geq 2$. We recall the following classical definition of uniform domains, \cite{JONES-1981, WALLIN-1991}.  

\begin{definition}\label{DefEDD}
	Let $\varepsilon > 0$. A bounded domain $\Omega\subset \mathbb{R}^n$ is called an \emph{$(\varepsilon,\infty)$-(uniform) domain} if for all $x, y \in \Omega$ there is a rectifiable curve $\gamma\subset \Omega$ with length $\ell(\gamma)$ joining $x$ to $y$ and satisfying
	\begin{enumerate}
		\item[(i)] $\ell(\gamma)\le \frac{|x-y|}{\varepsilon}$ and
		\item[(ii)] $d(z,\partial \Omega)\ge \varepsilon |x-z|\frac{|y-z|}{|x-y|}$ for $z\in \gamma$.
	\end{enumerate}
\end{definition}

By (a special case of) \cite[Theorem 1]{JONES-1981} any bounded $(\varepsilon,\infty)$-domain $\Omega\subset \mathbb{R}^n$ is a $W^{1,p}$-extension domain, and there is an extension operator $E_\Omega$ whose operator norm depends only on $n$, $p$ and $\varepsilon$, see also \cite{AHMNT,Rogers}. Theorem \ref{T:uniPoincare} and Corollary \ref{C:PoincareXOmega} below will also make use of a homogeneous version of this statement, \cite[Theorem 2]{JONES-1981}. In both cases the parameter $\varepsilon$ provides a quantitative control on the norm of extension operators.

For $(\varepsilon,\infty)$-domains $\Omega$ inside a bounded confinement $D$ and Borel measures $\mu$ supported inside $\overline{\Omega}$ and satisfying (\ref{E:upperdreg}) with the same constant we can obtain Poincar\'e inequalities with constants that are 'uniform', that is, independent of the particular shape of the domain $\Omega$ or the set $\Gamma$. 
A related result for the the case that $\Gamma\subset \partial\Omega$ and the functions in question have vanishing trace on $\Gamma$ had been shown in \cite[Theorem 10]{DRPT-2020}.

\begin{theorem}\label{T:uniPoincare}
Let $D\subset \mathbb{R}^n$ be a bounded Lipschitz domain, $1<p<+\infty$, $0\vee (n-p)<d\leq n$ and $\varepsilon>0$. Suppose that $\Omega\subset D$ is a $(\varepsilon,\infty)$-domain and that $\mu$ is a nonzero Borel measure on $\mathbb{R}^n$ with $\Gamma=\supp\mu\subset \overline{\Omega}$ and such that \eqref{E:upperdreg} holds. Then we have
\[\big\|u-\frac{1}{\mu(\Gamma)}\int_\Gamma \mathrm{Tr}_{\Omega,\Gamma} u\:d\mu\big\|_{L^p(\Omega)}\leq C(n,D,p,\varepsilon, d, c_d)\left\|\nabla u\right\|_{L^p(\Omega,\mathbb{R}^n)},\quad u\in W^{1,p}(\Omega),\]
with a constant $C(n,D,p,\varepsilon, d, c_d)$ depending only on $n,D,p,\varepsilon, d$ and $c_d$.
\end{theorem}

\begin{remark}
The constant $C(n,D,p,\varepsilon, d, c_d)$ does not depend on the particular choice of the measure $\mu$.
\end{remark}

Theorem \ref{T:uniPoincare} yields an equivalent norm on $W^{1,p}(\Omega)$ with uniform constants in the norm comparison. 

\begin{corollary}\label{C:equivnorms}
Let the hypotheses of Theorem \ref{T:uniPoincare} be in force. Then we have 
\[c^{-1}\left\|u\right\|_{W^{1,p}(\Omega)}\leq \left\|u\right\|_{W^{1,p}(\Omega),\mu}\leq c\:\left\|u\right\|_{W^{1,p}(\Omega)}, \quad u\in W^{1,p}(\Omega),\]
where 
\[\left\|u\right\|_{W^{1,p}(\Omega),\mu}:=\left(\left\|\nabla u\right\|_{L^p(\Omega,\mathbb{R}^n)}^p+\left\|\mathrm{Tr}_{\Omega,\Gamma} u\right\|_{L^p(\Gamma,\mu)}^p\right)^{1/p}\]
and $c>0$ is a constant depending only on $n,D,p,\varepsilon, d, c_d$ and $\mu(\Gamma)$, the dependency on $\mu(\Gamma)$ being continuous.
\end{corollary}

We formulate proofs of Theorem \ref{T:uniPoincare} and Corollary \ref{C:equivnorms} using three intermediate statements. In the sequel let $n$, $p$, $d$ and $c_d$ be as specified in Theorem \ref{T:uniPoincare}. 

For any $W^{1,p}$-extension domain $U\subset \mathbb{R}^n$ we write 
\begin{multline}\label{E:cone}
X(U):=\big\lbrace u\in W^{1,p}(U):\ \text{there is a Borel probability measure $\mu$ on $\mathbb{R}^n$}\\
 \text{with $\Gamma:=\supp\mu\subset \overline{U}$, satisfying (\ref{E:upperdreg}) and such that $\int_\Gamma \mathrm{Tr}_{U,\Gamma}u\:d\mu=0$}\big\rbrace.
\end{multline} 

The set $X(U)$ is a cone in the sense that for any $u\in X(U)$ and $\lambda\geq 0$ we have $\lambda u\in X(U)$. Moreover, $X(U)$ does not contain any constant except zero. 

\begin{remark}
In general $X(U)$ is not convex: For $U=(0,1)^2$, $u(x_1,x_2)=x_1$ and $u(x_1,x_2)=1-x_1$ we have 
$u_1, u_2\in X(U)$: Clearly $u_1,u_2\in W^{1,p}(U)$, and $\int u_1 d\mu_1=\int u_2 d\mu_2=0$ for the probability measures 
$\mu_1:=\mathcal{H}^1|_{\{0\}\times (0,1)}$ and $\mu_2:=\mathcal{H}^1|_{\{1\}\times (0,1)}$. 
However, $\frac12(u_1+u_2)=\frac12$ is not in $X(U)$.
\end{remark}

\begin{lemma}\label{L:weaklyclosed}
Let $D\subset \mathbb{R}^n$ be a bounded Lipschitz domain. Then the set $X(D)$ is weakly closed in $W^{1,p}(D)$.
\end{lemma}
\begin{proof}
Suppose that $(u_k)_k\subset X(D)$ converges to some $u$ weakly in $W^{1,p}(D)$. For each $k$ let $\mu_k$ be a probability measure corresponding to $u_k$ as in (\ref{E:cone}). By Banach-Alaoglu there is a sequence $(k_j)_j$ of indices such that 
$(\mu_{k_j})_j$ converges to some probability measure $\mu$ weakly on $\overline{D}$. By weak convergence $\Gamma:=\supp \mu$ is contained in $\overline{D}$, and by \cite[Proposition 2 (i)]{HINZ-2020} the measure $\mu$ satisfies (\ref{E:upperdreg}). Writing $\Gamma_k:=\supp \mu_k$ we obtain
\[\int_\Gamma \mathrm{Tr}_{D,\Gamma}u\:d\mu=\lim_j \int_{\Gamma_{k_j}} \mathrm{Tr}_{D,\Gamma_{k_j}}u_{k_j}\:d\mu=0\]
from Lemma \ref{L:traceconvergence}. Consequently $u\in X(D)$.
\end{proof}

We can now conclude a Poincar\'e inequality for elements of $X(D)$.

\begin{corollary}\label{C:PoincareXD}
Let $D\subset \mathbb{R}^n$ be a bounded Lipschitz domain. Then we have 
\begin{equation}\label{E:PoincareXD}
\left\|u\right\|_{L^p(D)}\leq c(n, D,p, d,c_d)\left\|\nabla u\right\|_{L^p(D,\mathbb{R}^n)},\quad u\in X(D),
\end{equation}
with $c(n,D,p,d,c_d)>0$ depending only on $n$, $D$, $p$, $d$ and $c_d$.
\end{corollary}

Apart from the use of $X(D)$ the argument is standard, see for instance \cite[Proposition 7.1]{Egert2015} or \cite[Section 5.8, Theorem 1]{EVANS-1994}.
\begin{proof}
Suppose that (\ref{E:PoincareXD}) does not hold. Then there is a sequence $(u_k)_k\subset X(D)$ such that 
$\frac{1}{k}\left\|u_k\right\|_{L^p(D)}\geq \left\|\nabla u_k\right\|_{L^p(D)}$.
Since $X(D)$ is a cone, we may divide each $u_k$ by its norm and therefore assume that 
\begin{equation}\label{E:asusual}
\left\|u_k\right\|_{L^p(D)}=1\quad \text{for all $k$.}
\end{equation}
Then obviously $\lim_k \nabla u_k=0$ in  $L^p(D,\mathbb{R}^n)$, and since $\sup_k \left\|u_k\right\|_{W^{1,p}(D)}<+\infty$ and $W^{1,p}(D)$ is reflexive, there is a subsequence $(u_{k_j})_j$ that converges to some $u$ weakly in $W^{1,p}(D)$. By Lemma \ref{L:weaklyclosed} we have $u\in X(D)$, and since $\nabla: W^{1,p}(D)\to L^p(D,\mathbb{R}^n)$ is continuous, $(\nabla u_{k_j})_j$ 
converges to $\nabla u$ weakly in $L^p(D,\mathbb{R}^n)$. But this means that $\nabla u=0$, so that $u\in X(D)$ is constant, hence $u=0$, and we see that $\lim_j u_{k_j}=0$ weakly in $W^{1,p}(D)$. By the Rellich-Kondrachov theorem for $D$ it now follows that $\lim_j u_{k_j}=0$ strongly in $L^p(D)$, what contradicts (\ref{E:asusual}).
\end{proof}

The preceding allows to conclude a Poincar\'e inequality simultaneously valid for any $(\varepsilon,\infty)$-domain $\Omega$ contained in $D$ with a uniform constant.

\begin{corollary}\label{C:PoincareXOmega}
Let $D\subset \mathbb{R}^n$ be a bounded Lipschitz domain and $\varepsilon>0$. There is a constant $C(n,D,p,\varepsilon, d,c_d)>0$ depending only on $n$, $D$, $p$, $\varepsilon$, $d$ and $c_d$ such that 
for any $(\varepsilon,\infty)$-domain $\Omega\subset D$ we have 
\begin{equation}\label{E:PoincareXOmega}
\left\|v\right\|_{L^p(\Omega)}\leq C(n,D,p,\varepsilon, d,c_d)\left\|\nabla v\right\|_{L^p(\Omega,\mathbb{R}^n)},\quad v\in X(\Omega).
\end{equation}
\end{corollary}

\begin{proof}
Given $v\in X(\Omega)$ we have $E_\Omega v\in X(D)$: If $\mu$ is a Borel probability measure wih support $\Gamma\subset\overline{\Omega}$ with respect to which $\mathrm{Tr}_{\Omega,\Gamma}v$ has integral zero, then also $\mathrm{Tr}_{D,\Gamma}E_\Omega v$ has integral zero with respect to $\mu$ because $\mathrm{Tr}_{\Omega,\Gamma}v=\mathrm{Tr}_{D,\Gamma}E_\Omega v$ $\mu$-a.e. on $\Gamma$ by \cite[Theorem 6.1]{Biegert2009}. Corollary \ref{C:PoincareXD} and \cite[Theorem 2]{JONES-1981} yield
\[\left\|v\right\|_{L^p(\Omega)}\leq \left\|E_\Omega v\right\|_{L^p(D)}\leq c(n,D,p,d,c_d)\left\|\nabla E_\Omega v\right\|_{L^p(D,\mathbb{R}^n)}\leq c(n,D,p,d,c_d)\:c(n,p,\varepsilon)\left\|\nabla v\right\|_{L^p(\Omega,\mathbb{R}^n)},\]
here $c(n,p,\varepsilon)>0$ is the norm of the extension operator $E_\Omega$ in \cite[Theorem 2]{JONES-1981}.
\end{proof}

Now Theorem \ref{T:uniPoincare} and Corollary \ref{C:equivnorms} follow easily.

\begin{proof}[Proof of Theorem \ref{T:uniPoincare}]
For any $u\in W^{1,p}(\Omega)$ the function $v:=u-\frac{1}{\mu(\Gamma)}\int_\Gamma \mathrm{Tr}_{\Omega,\Gamma} u\:d\mu$ is in $X(\Omega)$, and Corollary \ref{C:PoincareXOmega} yields the desired inequality.
\end{proof}

\begin{proof}[Proof of Corollary \ref{C:equivnorms}]
With $p<q<pd/(n-p)$ H\"older's inequality, Theorem \ref{ThTracecheap} and \cite[Theorem 1]{JONES-1981} yield
\[\left\|\mathrm{Tr}_{\Omega,\Gamma}u\right\|_{L^p(\Gamma,\mu)}\leq \mu(\Gamma)^{1/p-1/q}\left\|\mathrm{Tr}_{\Omega,\Gamma}u\right\|_{L^q(\Gamma,\mu)}\leq \mu(\Gamma)^{1/p-1/q}c_{\mathrm{Tr}}(n,p,d,c_d,q)c(n,p,\varepsilon)\left\|u\right\|_{W^{1,p}(\Omega)}.\]
On the other hand, Theorem \ref{T:uniPoincare} implies
\[\left\|u\right\|_{L^p(\Omega)}\leq c(n,D,p,\varepsilon,d,c_d)\left\|\nabla u\right\|_{L^p(\Omega,\mathbb{R}^n)}+\mu(\Gamma)^{-1}\left\|\mathrm{Tr}_{\Omega,\Gamma}u\right\|_{L^p(\Gamma,\mu)}.\] 
\end{proof}

As an auxiliary observation we state the following version of Corollary \ref{C:PoincareXD}, it resembles the classical Poincar\'e inequality for balls, \cite[Section 5.8, Theorem 2]{EVANS-1994}.
\begin{corollary}
For each $x\in\mathbb{R}^n$ and $r>0$ we have 
\[\left\|u\right\|_{L^p(B(x,r))}\leq c(n,p,d,c_d)\:r\:\left\|\nabla u\right\|_{L^p(B(x,r),\mathbb{R}^n)}, \quad u\in X(B(x,r)).\]
\end{corollary}
\begin{proof}
We use the shortcut notation $B:=B(x,r)$ and set $\ell(y):=x+ry$, $y\in\mathbb{R}^n$. Suppose that $u\in W^{1,p}(B)$ and $\mu$ is a probability measure with $\Gamma=\supp\mu\subset \overline{B}$ such that $\int_\Gamma u\:d\mu=0$. Then $v:=u\circ \ell$ is an element of $W^{1,p}(B(0,1))$. Clearly $w:=E_Bu\circ \ell$ is an extension of $v$ to an element of $W^{1,p}(\mathbb{R}^n)$, and the pointwise redefinition $\widetilde{w}$ of $w$ as in (\ref{E:redefinition}) is $(E_B u)^\sim\circ \ell$, note that
\begin{equation}\label{E:averages}
\frac{1}{\lambda^n(B(0,\varepsilon))}\int_{B(y,\varepsilon)}w(z)dz=\frac{1}{\lambda^n(B(0,r\varepsilon))}\int_{B(x+ry,r\varepsilon)}(E_B u)(\zeta)\:d\zeta
\end{equation}
for any $\varepsilon>0$. The push-forward $\nu:=(\ell^{-1})_\ast \mu$ of $\mu$ under $\ell^{-1}$ has support $\Gamma':=\supp\nu\subset \overline{B(0,1)}$, and since by \cite[Theorem 6.1]{Biegert2009} (or \cite[Theorem 1]{WALLIN-1991}) we have $\mathrm{Tr}_{B(0,1),\Gamma'}v=\mathrm{Tr}_{\Gamma'}w$, equality (\ref{E:averages}) implies 
\[\int_{\Gamma'}\mathrm{Tr}_{B(0,1),\Gamma'}v\:d\nu=\int_{\Gamma'}\widetilde{w}\:d\nu=\int_{\Gamma'}(E_B u)^\sim d\mu=\int_\Gamma \mathrm{Tr}_{B,\Gamma}u\:d\mu=0.\]
Consequently $v\in X(B(0,1))$, and by Corollary \ref{C:PoincareXD}, $\left\|v\right\|_{L^p(B(0,1))}\leq c(n,p,d,c_d)\:\left\|\nabla v\right\|_{L^p(B(0,1)),\mathbb{R}^n)}$, what yields the statement.
\end{proof}

\section{Generalized boundary value problems}\label{S:bvp}

To discuss generalized boundary value problems we specialize again to the case $p=2$. Let $\Omega$ be a $W^{1,2}$-extension domain (not necessarily bounded). Suppose that $\mu$ is a nonzero finite Borel measure whose support $\supp\mu=\Gamma$ is a subset of $\overline{\Omega}$ and which satisfies \eqref{E:upperdreg} with some $0\vee (n-2)<d\leq n$. We write 
\begin{equation}\label{E:VOmegaGamma}
V(\Omega,\Gamma):=\left\lbrace v\in W^{1,2}(\Omega):\ (E_\Omega v)^\sim =0\ \text{$W^{1,2}(\mathbb{R}^n)$-q.e. on $\Gamma$}\right\rbrace
\end{equation}
for the closed subspace of $W^{1,2}(\Omega)$ consisting of elements with zero trace on $\Gamma$ in the $W^{1,2}(\mathbb{R}^n)$-q.e. sense. Let $\alpha\geq 0$, $f\in L^2(\Omega)$ and $\varphi\in \mathrm{Tr}_{\Gamma}(W^{1,2}(\mathbb{R}^n))$. We make the \emph{generalized Dirichlet problem} 
\begin{equation}\label{E:abstractDirichlet}
\begin{cases}
\quad (\alpha-\Delta) u&=f\quad \text{on $\Omega\setminus \Gamma$}\\
\quad \mathrm{Tr}_{\Omega,\Gamma} u & =\varphi\quad \text{on $\Gamma$}
\end{cases}
\end{equation}
rigorous by saying that $u\in W^{1,2}(\Omega)$ is a \emph{weak solution} to (\ref{E:abstractDirichlet}) if
$\mathrm{Tr}_{\Omega,\Gamma} u  =\varphi$ holds $W^{1,2}(\mathbb{R}^n)$-q.e. on $\Gamma$ and we have 
\begin{equation}\label{E:weaksolD}
\int_{\Omega}\nabla u \nabla v\ dx+\alpha \int_\Omega uv\ dx=\int_{\Omega} f\:v\ dx, \quad v\in V(\Omega,\Gamma).
\end{equation}


\begin{remark}\mbox{}
\begin{enumerate}
\item[(i)] For $\Omega=\mathbb{R}^n$ and $f\equiv 0$ problem (\ref{E:abstractDirichlet}) is just problem (\ref{E:Dirichletfortrace}) with $\beta=1$. 
\item[(ii)] A function $u\in W^{1,2}(\Omega)$ is a weak solution to (\ref{E:abstractDirichlet}) if and only if the first equation in (\ref{E:abstractDirichlet}) holds in $\mathcal{D}'(\Omega\setminus \Gamma)$ and the second holds $W^{1,2}(\mathbb{R}^n)$-q.e. If $u$ is a weak solution, then by the quality of $f$ the first equation holds in $L^2(\Omega\setminus \Gamma)$ and by the quality of $\varphi$ the second holds in $L^2(\Gamma,\mu)$. 
\item[(iii)] Suppose that $u$ is a weak solution to (\ref{E:abstractDirichlet}) and that $\mu'$ is a finite Borel measure equivalent to $\mu$ and satisfying (\ref{E:upperdreg}) (possibly with a different constant $c_d'>0$). Then, trivially, $u$ is also a weak solution to \eqref{E:abstractDirichlet} based on $\mu'$ in place of $\mu$.
\end{enumerate}

\end{remark}

\begin{proposition}\label{P:exDirichlet}
Let $\Omega$ be a $W^{1,2}$-extension domain. Let $0\vee(n-2)<d\leq n$ and suppose that $\mu$ is a nonzero finite Borel measure satisfying \eqref{E:upperdreg} whose support $\Gamma=\supp\mu$ is a subset of $\overline{\Omega}$. Suppose that $\alpha\geq 0$, $f\in L^2(\Omega)$ and $\varphi \in \mathrm{Tr}_\Gamma(W^{1,2}(\mathbb{R}^n))$. 

If $\alpha>0$ or $\Omega$ is bounded, then there is a unique weak solution $u\in W^{1,2}(\Omega)$  to \eqref{E:abstractDirichlet}. 
\end{proposition}


\begin{proof}
If $\alpha>0$ then the result follows using a similar argument as in Proposition \ref{P:fracext} and Lax-Milgram. If $\alpha=0$ and $\Omega$ is bounded we can use Poincar\'e's inequality for functions with zero trace on $\Gamma$ to obtain the necessary coercivity bound, see for instance \cite[Section 1.3]{GIRAULT-1986} or \cite[Section 22.2a]{ZEIDLER}. 
\end{proof}

To discuss other types of boundary value problems we introduce generalized normal derivatives by a variant of a standard procedure. We say that $u\in W^{1,2}(\Omega)$ is in $\mathcal{D}_{\Omega,\Gamma}(\Delta)$ if there is some $f\in L^2(\Omega)$ such that 
\[\int_\Omega \nabla u \nabla v\ dx=-\int_\Omega f\ v\ dx,\quad v\in V(\Omega,\Gamma).\]
For such $u$ we set $\Delta u:=f$. For all $u\in \mathcal{D}_{\Omega,\Gamma}(\Delta)$ we can define  a bounded linear functional $\frac{\partial u}{\partial n_\Gamma} \in (\mathrm{Tr}_{\Gamma}(W^{1,2}(\mathbb{R}^n)))'$ by the identity
		\begin{multline}\label{FracGreen}
				\langle \frac{\partial u}{\partial n_\Gamma}, 
		\psi \rangle_{(\mathrm{Tr}_{\Gamma}(W^{1,2}(\mathbb{R}^n)))',\mathrm{Tr}_{\Gamma}(W^{1,2}(\mathbb{R}^n))}\\
		=\int_\Omega H_{\Gamma,\Omega} \psi\: \Delta u\:dx + \int_\Omega \nabla H_{\Gamma,\Omega} \psi \: \nabla u  \:dx, \quad \psi \in \mathrm{Tr}_{\Gamma}(W^{1,2}(\mathbb{R}^n)),
		\end{multline}
here $H_{\Gamma,\Omega}$ is defined as in (\ref{E:harmext}). We refer to $\frac{\partial u}{\partial n_\Gamma}$ as the \emph{generalized normal derivative of $u$ on $\Gamma$}. 

\begin{remark}\mbox{}
\begin{enumerate}
\item[(i)] The right hand side of equality (\ref{FracGreen}) remains unchanged if the function $H_{\Gamma,\Omega}\psi$ is replaced by an arbitrary element $w$ of $W^{1,2}(\Omega)$ satisfying $\mathrm{Tr}_{\Omega,\Gamma}w=\psi$. This is due to  the orthogonal decomposition of $W^{1,2}(\Omega)$ into $V(\Omega,\Gamma)$ plus $\{H_{\Gamma,\Omega}\psi: \psi \in \mathrm{Tr}_{\Gamma}(W^{1,2}(\mathbb{R}^n))\}$, which follows from (\ref{E:weaksolD}), see \cite[Section 2.3]{FOT94}.  
\item[(ii)] If $\Omega$ is bounded and $\Gamma=\partial\Omega$, then $\frac{\partial u}{\partial n_{\Gamma}}$ is the generalized normal derivative $\frac{\partial u}{\partial n}$ on $\partial\Omega$ as discussed in \cite[Theorem 7]{HINZ-2020}.
\item[(iii)] If $\mu'$ is another Borel measure satisfying \eqref{E:upperdreg} (with a possibly different constant $c_d'>0$) and $\mu'$ is equivalent to $\mu$, then by Proposition \ref{P:tracespacedoesnotdepend} the generalized normal derivative $\frac{\partial u}{\partial n_\Gamma}$ of $u$ on $\Gamma$ based on $\mu'$ is the same as the one based on $\mu$.
 \end{enumerate}
\end{remark}
	
Now let $\alpha\geq 0$, let $\gamma\in L^\infty(\Gamma,\mu)$ be a Borel function, $f\in L^2(\Omega)$ and $\varphi\in \mathrm{Tr}_{\Gamma}(W^{1,2}(\mathbb{R}^n))$. We make the \emph{generalized Robin problem} 
\begin{equation}\label{E:abstractRobin}
\begin{cases}
\qquad (\alpha-\Delta) u &= f\quad \text{in $\Omega\setminus \Gamma$} \\
\frac{\partial u}{\partial n_\Gamma} +\gamma \operatorname{Tr}_{\Omega, \Gamma}u &=\varphi\quad \text{on $\Gamma$}.
\end{cases}
\end{equation}
rigorous by saying that $u\in W^{1,2}(\Omega)$ is a \emph{weak solution} to (\ref{E:abstractRobin}) if $u$ solves 
\begin{equation}\label{E:weaksolR}
\int_\Omega \nabla u\:\nabla v\ dx+\alpha\int_\Omega uv\ dx+\int_\Gamma \gamma\: \mathrm{Tr}_{\Omega,\Gamma} u\: \mathrm{Tr}_{\Omega,\Gamma} v\ d\mu=\int_\Omega f\:v\ dx+\int_\Gamma\varphi\:\mathrm{Tr}_{\Omega,\Gamma} v\ d\mu,\quad v\in W^{1,2}(\Omega).\end{equation}
Note that the special case $\gamma\equiv 0$ makes (\ref{E:abstractRobin}) an \emph{generalized Neumann problem}. 

\begin{remark}\label{R:weakvsL2} If $\alpha>0$ or $\gamma\not\equiv 0$, then we can observe the following.
\begin{enumerate}
\item[(i)] A function $W^{1,2}(\Omega)$ is a weak solution to (\ref{E:abstractRobin}) if and only if it satisfies the first equation in (\ref{E:abstractRobin}) in $\mathcal{D}'(\Omega\setminus \Gamma)$ and the second in $(\mathrm{Tr}_{\Gamma}(W^{1,2}(\mathbb{R}^n)))'$. If $u$ is a weak solution, then by the quality of $f$ also the first equation holds in $L^2(\Omega\setminus \Gamma)$ and by the quality of $\varphi$ the second equation holds in $L^2(\Gamma,\mu)$ and in particular, $\frac{\partial u}{\partial n_\Gamma}\in L^2(\Gamma,\mu)$.
\item[(ii)] Suppose that $u$ is a weak solution to (\ref{E:abstractRobin}) and that $\mu'$ is a finite Borel measure equivalent to $\mu$ and satisfying (\ref{E:upperdreg}) (possibly with a different constant $c_d'>0$). Then $u$ is also a weak solution to \eqref{E:abstractRobin} based on $\mu'$ in place of $\mu$.
\end{enumerate}
\end{remark}

\begin{remark}
Suppose that $\alpha=0$, $\gamma\equiv 0$ and $\Omega$ is bounded. If $f$ and $\varphi$ satisfy the condition
\begin{equation}\label{E:compatible}
\int_\Omega f\ dx+\int_\Gamma \varphi\ d\mu=0,
\end{equation}
then similar observations as in Remark \ref{R:weakvsL2} (i) remain true. If (\ref{E:compatible}) is not satisfied, one can  investigate the variational problem (\ref{E:weaksolR}) with $W^{1,2}(\Omega)$ replaced by 
\[V(\Omega)=\Big\lbrace u\in W^{1,2}(\Omega):\ \int_\Omega u\:dx=0\Big\rbrace,\]
which, endowed with the norm $\left\|u\right\|_{V(\Omega)}:=\left\|\nabla u\right\|_{L^2(\Omega,\mathbb{R}^n)}$, is a Hilbert space. However, solutions of this variational problem may not directly correspond to (\ref{E:abstractRobin}). An analog of Remark \ref{R:weakvsL2} (ii) remains true in either case. 
\end{remark}

\begin{proposition}\label{P:exRobin}
Let $\Omega$ be a $W^{1,2}$-extension domain. Let $0\vee(n-2)<d\leq n$ and suppose that $\mu$ is a nonzero finite Borel measure satisfying \eqref{E:upperdreg} whose support $\Gamma=\supp\mu$ is a subset of $\overline{\Omega}$. Suppose that $\alpha\geq 0$, $\gamma\in L^\infty(\Gamma,\mu)$ is nonnegative, $f\in L^2(\Omega)$ and $\varphi \in \mathrm{Tr}_\Gamma(W^{1,2}(\mathbb{R}^n))$. 

If $\alpha>0$, then there is a unique weak solution $u$  to \eqref{E:abstractRobin}. If $\alpha=0$, $\Omega$ is bounded and $\gamma$ is bounded away from zero, then there is a unique weak solution $u$  to \eqref{E:abstractRobin}.  If $\alpha=0$, $\Omega$ is bounded, $\gamma\equiv 0$ and (\ref{E:compatible}) holds, then there is a weak solution $u$  to \eqref{E:abstractRobin}, and it is unique in $V(\Omega)$.
\end{proposition}

\begin{proof}
For $\alpha>0$ one can again apply Lax-Milgram in $W^{1,2}(\Omega)$ directly. For $\alpha\equiv 0$, $\Omega$ bounded and $\gamma$ bounded away from zero Corollary \ref{C:equivnorms}, applied with $D$ being a large enough ball, ensures the coercivity of $u\mapsto \int_\Omega|\nabla u|^2 dx+\int_\Gamma \gamma (\mathrm{Tr}_{\Omega,\Gamma}u)^2 d\mu$. See for instance \cite[Section 22.2g]{ZEIDLER}.
For $\alpha\equiv 0$, $\Omega$ bounded and $\gamma\equiv 0$ we can again find a large ball $D$ containing $\Omega$, and Rellich-Kondrachov for $D$ implies Rellich-Kondrachov for $\Omega$ and therefore a Poincar\'e inequality for $\Omega$ by the standard argument, see \cite[Proposition 7.1]{Egert2015} or \cite[Section 5.8, Theorem 1]{EVANS-1994}. In particular, we have $\left\|u\right\|_{L^2(\Omega)}\leq c\:\left\|\nabla u\right\|_{L^2(\Omega,\mathbb{R}^n)}$ for all $u$ from $V(\Omega)$. Using Lax-Milgram with $V(\Omega)$ one arrives at the desired result, see for instance \cite[Section 1.4]{GIRAULT-1986} or \cite[Section 22.2f]{ZEIDLER}.
\end{proof}

Given $\alpha\geq 0$ and nonnegative $\gamma\in L^\infty(\Gamma,\mu)$ we consider the quadratic form 
\begin{equation}\label{E:quadraticform}
\mathcal{E}_{\alpha,\gamma}(u)=\int_\Omega |\nabla u|^2 dx+\alpha\int_\Omega u^2 dx+\int_\Gamma \gamma (\mathrm{Tr}_{\Omega,\Gamma}u)^2d\mu,\quad u\in W^{1,2}(\Omega).
\end{equation}  

The following is a special case of a well-known fact, see for instance \cite[Corollary 1.2]{GIRAULT-1986} or \cite[Section 22.2]{ZEIDLER}.

\begin{proposition}\label{P:weaksol}
Let $\Omega$ be a $W^{1,2}$-extension domain. Let $0\vee(n-2)<d\leq n$ and suppose that $\mu$ is a nonzero finite Borel measure satisfying \eqref{E:upperdreg} whose support $\Gamma=\supp\mu$ is a subset $\overline{\Omega}$. Suppose that $\alpha\geq 0$, $\gamma\in L^\infty(\Gamma,\mu)$ is nonnegative, $f\in L^2(\Omega)$ and $\varphi \in \mathrm{Tr}_\Gamma(W^{1,2}(\mathbb{R}^n))$. 
\begin{enumerate}
\item[(i)] If $\alpha>0$ or $\Omega$ is bounded, then $u\in V(\Omega,\Gamma)$ is a weak solution to (\ref{E:abstractDirichlet}) with $\varphi\equiv 0$ if and only if it minimizes the functional $v\mapsto \frac12 \mathcal{E}_{\alpha,0}(v)-\int_\Omega fv\:dx$ on $V(\Omega, \Gamma)$.
\item[(ii)] Suppose that $\gamma$ is bounded away from zero. If $\alpha>0$ or $\Omega$ is bounded, then $u\in W^{1,2}(\Omega)$ is a weak solution to (\ref{E:abstractRobin}) if and only if it minimizes the functional $v\mapsto \frac12 \mathcal{E}_{\alpha,\gamma}(v)-\int_\Omega fv\:dx-\int_\Gamma \varphi \mathrm{Tr}_{\Omega,\Gamma}v\:d\mu$ on $W^{1,2}(\Omega)$.
\item[(iii)] Suppose that $\gamma\equiv 0$. If $\alpha>0$, then $u\in W^{1,2}(\Omega)$ is a weak solution to (\ref{E:abstractRobin}) if and only if it minimizes the functional $v\mapsto \frac12 \mathcal{E}_{\alpha,0}(v)-\int_\Omega fv\:dx-\int_\Gamma \varphi \mathrm{Tr}_{\Omega,\Gamma}v\:d\mu$ on $W^{1,2}(\Omega)$. If $\alpha=0$, $\Omega$ is bounded and (\ref{E:compatible}) holds, then a member $u$ of the space $V(\Omega)$ is a weak solution to (\ref{E:abstractRobin}) if and only if it minimizes the same functional on $V(\Omega)$.
\end{enumerate}
\end{proposition}

Partially following \cite[Section VIII.6]{REED-SIMON-1980} we refer to a closed, densely defined and nonnegative definite symmetric bilinear form on a Hilbert space as a a \emph{closed quadratic form}. 

\begin{corollary}\label{C:closed}
Under the respective hypotheses on $\Omega$, $\mu$, $\alpha$ and $\gamma$ formulated in Proposition \ref{P:weaksol} (i), (ii) and (iii), respectively, the forms 
\[(\mathcal{E}_{\alpha,0}, V(\Omega,\Gamma)),\quad (\mathcal{E}_{\alpha,\gamma},W^{1,2}(\Omega))\quad \text{and}\quad (\mathcal{E}_{\alpha,0}, W^{1,2}(\Omega))\] 
are closed quadratic forms on $L^2(\Omega)$. 
\end{corollary}

Let $D\subset \mathbb{R}^n$ be a domain. Given a subdomain $\Omega\subset D$ we may view $L^2(\Omega)$ as a closed subspace of $L^2(D)$ by continuation by zero. Then 
\[P_\Omega f:=\mathbf{1}_\Omega f\] 
is the orthogonal projection from $L^2(D)$ onto $L^2(\Omega)$. We denote the spaces of all $f\in L^2(D)$ such that $P_\Omega f\in W^{1,2}(\Omega)$ respectively $P_\Omega f\in V(\Omega,\Gamma)$ by $W^{1,2}(\Omega)_D$ respectively $V(\Omega,\Gamma)_D$. If it is a priori clear that we are talking about an element $f$ of $L^2(D)$, then we also just write $W^{1,2}(\Omega)$ respectively $V(\Omega,\Gamma)$; this is in line with our general policy to understand the notations $v\in W^{1,2}(\Omega)$ and $\mathcal{E}(v)$ as $v|_\Omega\in W^{1,2}(\Omega)$ and $\mathcal{E}(v|_\Omega)$ if $\mathcal{E}$ is defined on $W^{1,2}(\Omega)$ and $v$ is defined on a larger set containing $\Omega$. 

To point out the dependency of the quadratic forms defined in (\ref{E:quadraticform}) on the domain $\Omega\subset D$ and the measure $\mu$, we also write 
\[\mathcal{E}_{\alpha,\gamma}^{\Omega,\mu}:=\mathcal{E}_{\alpha,\gamma}.\]
It is straightforward to see that Corollary \ref{C:closed} implies the following.

\begin{corollary}\label{C:closedonD}
Let $D\subset \mathbb{R}^n$ be a domain and $\Omega\subset D$ a $W^{1,2}$-extension domain. Under the respective hypotheses on $\Omega$, $\mu$, $\alpha$ and $\gamma$ formulated in Proposition \ref{P:weaksol} (i), (ii) and (iii), the forms 
\[(\mathcal{E}_{\alpha,0}^{\Omega,\mu}\circ P_\Omega,V(\Omega,\Gamma)_D),\quad (\mathcal{E}_{\alpha,\gamma}^{\Omega,\mu}\circ P_\Omega,W^{1,2}(\Omega)_D)\quad \text{ and }\quad (\mathcal{E}_{\alpha,0}^{\Omega,\mu}\circ P_\Omega, W^{1,2}(\Omega)_D)\] 
are closed quadratic forms on $L^2(D)$. 
\end{corollary}

We absorb the projection operators into the notation and set, under the respective hypotheses stated in Corollary \ref{C:closedonD},
\begin{equation}\label{E:dependondomain}
\hat{\mathcal{E}}_{\alpha,\gamma}^{\Omega,\mu}:=\mathcal{E}_{\alpha,\gamma}\circ P_\Omega.
\end{equation}

For any $\alpha>0$ we can define a resolvent operator $\hat{G}_{\alpha,0}^{\Omega,\mu,D}:L^2(D)\to V(\Omega,\Gamma)_D$  by the identity
\begin{equation}\label{E:GreenD}
\hat{\mathcal{E}}_{\alpha,0}^{\Omega,\mu}(\hat{G}_{\alpha,0}^{\Omega,\mu,D}f,v)=\int_D f\:v\:dx,\quad f\in L^2(D),\ v\in V(\Omega,\Gamma)_D;
\end{equation}
if $\Omega$ is bounded, then also $\alpha=0$ may be chosen. If $\gamma$ is bounded away from zero then for any $\alpha>0$ we can define a resolvent operator $\hat{G}_{\alpha,\gamma}^{\Omega,\mu,R}:L^2(D)\to W^{1,2}(\Omega)_D$ by the identity
\begin{equation}\label{E:GreenR}
\hat{\mathcal{E}}_{\alpha,\gamma}^{\Omega,\mu}(\hat{G}_{\alpha,\gamma}^{\Omega,\mu,R}f,v)=\int_D f\:v\:dx,\quad f\in L^2(D),\ v\in W^{1,2}(\Omega)_D;
\end{equation}
again we can permit $\alpha=0$ if $\Omega$ is bounded. If $\gamma\equiv 0$, then for any $\alpha>0$ we can define a resolvent operator $\hat{G}_{\alpha,0}^{\Omega,\mu,N}:L^2(D)\to W^{1,2}(\Omega)_D$ by the identity
\begin{equation}\label{E:GreenN}
\hat{\mathcal{E}}_{\alpha,0}^{\Omega,\mu}(\hat{G}_{\alpha,0}^{\Omega,\mu,N}f,v)=\int_D f\:v\:dx,\quad f\in L^2(D),\ v\in W^{1,2}(\Omega)_D.
\end{equation}

In the following we will also use the shortcut notation $\hat{G}_{\alpha,\gamma}^{\Omega,\mu,\ast}$, where $\ast$ stands for $D$, $R$ or $N$, the meaning will be clear from the context.

To clarify the relationship of the operators $\hat{G}_{\alpha,\gamma}^{\Omega,\mu,\ast}$ and the original problems (\ref{E:abstractDirichlet}) and (\ref{E:abstractRobin}) on the domain $\Omega$, let $G_{\alpha,\gamma}^{\Omega,\mu,\ast}$ denote the resolvent operator with parameter $\alpha> 0$ uniquely associated with $\mathcal{E}_{\alpha,\gamma}^{\Omega,\mu}$ on $L^2(\Omega)$.

\begin{proposition}\label{P:consistent}
Let $D\subset \mathbb{R}^n$ be a domain and $\Omega\subset D$ a $W^{1,2}$-extension domain. Suppose that $0\vee(n-2)<d\leq n$ and that $\mu$ is a nonzero finite Borel measure satisfying \eqref{E:upperdreg} whose support $\Gamma=\supp\mu$ is a subset $\overline{\Omega}$. Let $\alpha$ and $\gamma$ be as specified in (\ref{E:GreenD}), (\ref{E:GreenR}) or (\ref{E:GreenN}). 

Then we have 
\begin{equation}\label{E:flip}
P_\Omega\circ \hat{G}_{\alpha,\gamma}^{\Omega,\mu,\ast} = G_{\alpha,\gamma}^{\Omega,\mu,\ast}\circ P_\Omega.
\end{equation}

In particular, if $f\in L^2(D)$, then 
\[P_\Omega\circ \hat{G}_{\alpha,\gamma}^{\Omega,\mu,D} f,\quad  P_\Omega\circ \hat{G}_{\alpha,\gamma}^{\Omega,\mu,R}f\quad \text{and}\quad P_{\Omega}\circ \hat{G}_{\alpha,\gamma}^{\Omega,\mu,N}f\]
are the continuations by zero of the unique weak solutions for the case of zero boundary data $\varphi\equiv 0$ to (\ref{E:abstractDirichlet}), to (\ref{E:abstractRobin}) with $\gamma$ bounded away from zero and  to (\ref{E:abstractRobin}) with $\gamma\equiv 0$ and $\alpha>0$, respectively. 
\end{proposition}

\begin{proof} 
Let $H$ be a Hilbert space, $K\subset H$ a closed subspace and $P$ the orthogonal projection onto $K$. Suppose that $(Q,\mathcal{D}(Q))$ is a closed quadratic form on $K$ and and $G_\alpha$ the associated $\alpha$-resolvent operator, defined by the identity $Q_\alpha(G_\alpha f,g)=\left\langle f,g\right\rangle_K$, $f\in K$, $g\in \mathcal{D}(Q)$. Let $(\hat{Q},\mathcal{D}(\hat{Q}))$ be defined by $\mathcal{D}(\hat{Q}):=\left\lbrace f\in H: Pf\in \mathcal{D}(Q)\right\rbrace$ and $\hat{Q}:=Q\circ P$. It is straightforward to see that $(\hat{Q},\mathcal{D}(\hat{Q}))$ is a closed quadratic form on $H$. Let $\hat{G}_\alpha$ be its
$\alpha$-resolvent operator, that is, $\hat{Q}_\alpha(\hat{G}_\alpha f,g)=\left\langle f,g\right\rangle_H$, $f\in H$, $g\in \mathcal{D}(\hat{Q})$. Then 
\[Q_\alpha(P(\hat{G}_\alpha f),Pg)=\hat{Q}_\alpha(\hat{G}_\alpha f,Pg)=\left\langle f,Pg\right\rangle_H=\left\langle Pf,Pg\right\rangle_H=\left\langle Pf,Pg\right\rangle_K=Q_\alpha(G_\alpha(Pf),Pg)\]
for all $f\in H$ and $g\in \mathcal{D}(\hat{Q})$.
This implies that $P(\hat{G}_\alpha f)=G_\alpha(Pf)$ in $\mathcal{D}(Q)$ for all $f\in H$ and therefore (\ref{E:flip}). The remaining statements are a consequence.
\end{proof}

\begin{remark}\label{R:quasiinv}\mbox{}
\begin{enumerate}
\item[(i)] Let $(\mathcal{L}^{\Omega,\mu,\ast}_\gamma,\mathcal{D}(\mathcal{L}^{\Omega,\mu,\ast}_\gamma))$ be the generator of $\mathcal{E}_{0,\gamma}^{\Omega,\mu}$ with the respective boundary conditions, that is, the unique non-positive definite self-adjoint operator on $L^2(\Omega)$ satisfying 
\[\mathcal{E}_{0,\gamma}^{\Omega,\mu}(u,\psi)=-\left\langle\mathcal{L}^{\Omega,\mu,\ast}_\gamma u,\psi\right\rangle_{L^2(\Omega)}\]
for all $u\in \mathcal{D}(\mathcal{L}^{\Omega,\mu,\ast}_\gamma)$ and all $\psi$ from $V(\Omega,\Gamma)$ (if $\ast=D$), $W^{1,2}(\Omega)$ (if $\ast=R$ or $N$). The operator in (\ref{E:flip}) is the quasi-inverse of $(\alpha-\mathcal{L}^{\Omega,\mu,\ast}_\gamma)\circ P_\Omega$ in the sense of \cite{Weidmann1984}.
\item[(ii)] The operator in (\ref{E:flip}) can also be written as 
\begin{equation}\label{E:PGP}
P_\Omega\circ \hat{G}_{\alpha,\gamma}^{\Omega,\mu,\ast}\circ P_\Omega = P_\Omega\circ G_{\alpha,\gamma}^{\Omega,\mu,\ast}\circ P_\Omega.
\end{equation}
\item[(iii)] The operator $P_{\Omega,0}f:=\mathbf{1}_\Omega f-\mathbf{1}_\Omega \int_\Omega f\:dx$ is the orthogonal projection from $L^2(\Omega)$ (or $L^2(D)$) onto $L^2(\Omega)_0:=\{g\in L^2(\Omega): \int_\Omega g\:dx=0\}$. The restriction of $\mathcal{L}^{\Omega,\mu,N}_0$ to $L^2(\Omega)_0$ has a bounded inverse $\mathring{G}_{0,0}^{\Omega,\mu,N}$, and $P_{\Omega,0}\circ \mathring{G}_{0,0}^{\Omega,\mu,N}\circ P_{\Omega,0}$ is a quasi-inverse of $(-\mathcal{L}^{\Omega,\mu,N}_0)\circ P_{\Omega,0}$.
\end{enumerate}
\end{remark}

In the case of a bounded 'confinement' domain $D$ we can observe the following compactness property of quasi-inverses; the statement is similar to \cite[Lemma 4.7]{BucurGiacominiTrebeschi2016}. 

\begin{lemma}\label{L:compactnessprop}
Let $D\subset \mathbb{R}^n$ be a bounded Lipschitz domain and $\Omega\subset D$ a $W^{1,2}$-extension domain. Suppose that $0\vee(n-2)<d\leq n$ and that $\mu$ is a nonzero  Borel measure satisfying \eqref{E:upperdreg} whose support $\Gamma=\supp\mu$ is a subset $\overline{\Omega}$. Suppose further that $\gamma\equiv 0$, $\alpha\geq 0$ and $\ast=D$, or that $\gamma$ is bounded away from zero, $\alpha\geq 0$ and $\ast=R$, or that $\gamma\equiv 0$, $\alpha>0$ and $\ast=N$.

If $(g_m)_{m=1}^\infty\subset L^2(D)$ converges weakly in $L^2(D)$ to some $g$, then there is a sequence $(m_k)_{k=1}^\infty$ with $m_k\uparrow +\infty$ such that 
\begin{equation}\label{E:resconvsingle}
\lim_{k\to \infty} P_\Omega\circ \hat{G}_{\alpha,\gamma}^{\Omega,\mu,\ast}g_{m_k}=P_\Omega\circ \hat{G}_{\alpha,\gamma}^{\Omega,\mu,\ast}g\quad \text{in $L^2(D)$}.
\end{equation}
\end{lemma}

\begin{proof}
Let 
\[v_m:=P_\Omega\circ \hat{G}_{\alpha,\gamma}^{\Omega,\mu,\ast}g_m\quad \text{ and}\quad v:=P_\Omega\circ \hat{G}_{\alpha,\gamma}^{\Omega,\mu,\ast}g.\] 
Then, as observed in Proposition \ref{P:consistent}, $u_m=v_m|_{\Omega_m}$ and $u=v|_\Omega$
are the unique weak solutions to (\ref{E:abstractDirichlet}) or (\ref{E:abstractRobin}) with zero boundary data on $\Omega_m$ with $f=g_m$ and on $\Omega$ with $f=g$. By the symmetry and boundedness of $\hat{G}_{\alpha,\gamma}^{\Omega,\mu,\ast}$ we see that $\lim_{m\to \infty} v_m=v$ weakly in $L^2(D)$ and weakly in $L^2(\Omega)$. On the other hand, we can again use the fact that they may also serve as a test function in (\ref{E:weaksolD}) respectively (\ref{E:weaksolR}). This implies that 
\[\left\|v_m\right\|_{W^{1,2}(\Omega)}^2 \leq c\:\mathcal{E}_{\alpha,\gamma}^{\Omega,\mu}(v_m)=c\int_\Omega g_mv_m\:dx\leq c\left(\sup_m \left\|g_m\right\|_{L^2(D)}\right)\left\|v_m\right\|_{L^2(\Omega)} \]
with a constant $c>0$ independent of $m$ and therefore  $\sup_m \left\|v_m\right\|_{W^{1,2}(\Omega)}<+\infty$.
Since $\Omega$ is a $W^{1,2}$-extension domain, it follows that $\sup_m \left\|E_\Omega(v_m|_\Omega)\right\|_{W^{1,2}(D)}<+\infty$, and by Rellich-Kondrachov, applied to $D$, there is some $v^\ast\in L^2(\Omega)$ such that along a sequence $(m_k)_k$ with $m_k\uparrow +\infty$ we have $\lim_{k\to \infty} E_\Omega(v_{m_k}|_\Omega)=v^\ast$ in $L^2(D)$.
Consequently $\lim_{k\to \infty} v_{m_k}=v^\ast=v$ in $L^2(\Omega)$, and since $v$ and all $v_{m_k}$ vanish outside $\Omega$, the result follows.
\end{proof}

\begin{remark}
A similar statement could also be shown for the quasi-inverses in Remark \ref{R:quasiinv} (iii).
\end{remark}

\section{Stability under convergence of domains and measures}\label{S:Mosco}

In this section we assume that $n\geq 2$ and $D\subset \mathbb{R}^n$ is a bounded Lipschitz domain. Suppose that $\Omega\subset D$ is a $W^{1,2}$-extension domain and $\mu$ is a nonzero Borel measure satisfying (\ref{E:upperdreg}) whose support $\Gamma=\supp\mu$ is a subset of $\overline{\Omega}$. Let $\alpha\geq 0$ and let $\gamma:\overline{D}\to [0,+\infty)$ be a bounded Borel function. We write $\mathcal{E}_{\alpha,\gamma}^{\Omega,\mu}$ denotes the quadratic form $\mathcal{E}_{\alpha,\gamma}$ defined in (\ref{E:quadraticform}) for this fixed domain $\Omega$ and measure $\mu$. We define an \emph{energy functional}  $J(\Omega,\mu)$ on $L^2(D)$ by  
\begin{equation}\label{extJ--} 
J(\Omega,\mu)(v)  
=\begin{cases}
\mathcal{E}_{\alpha,\gamma}^{\Omega,\mu}(v), \ v\in W^{1,2}(\Omega) ,\\
+\infty  , \ \ \ v\notin W^{1,2}(\Omega).
\end{cases}
\end{equation}

If we are given sequences of such domains and measures and these sequences converge in suitable ways, then we can observe the convergence of the associated energy functionals. To formulate the result we recall the corresponding notions of convergence.

For any closed set $K\subset \mathbb{R}^n$ and $\varepsilon>0$ we write $(K)_\varepsilon=\{x\in\mathbb{R}^n: d(x,K)\leq \varepsilon\}$ for its closed (outer) $\varepsilon$-parallel set. Recall that the \emph{Hausdorff distance} between two closed sets $K_1, K_2\subset \mathbb{R}^n$ is defined as 
\[d^H(K_1,K_2):=\inf\{\varepsilon>0: K_1\subset (K_2)_\varepsilon \text{ and } K_2\subset (K_1)_\varepsilon\}.\]  
A sequence $(K_m)_m$ of closed sets $K_m\subset \mathbb{R}^n$ is said to \emph{converge} to a closed set $K\subset \mathbb{R}^n$ \emph{in the Hausdorff sense} if $\lim_{m\to \infty} d^H(K_m,K)=0$.
A sequence $(\Omega_m)_m$ of open sets $\Omega_m\subset D$ is said to \emph{converge} to an open set $\Omega\subset D$ \emph{in the Hausdorff sense} if 
\[d^H(\overline{D}\setminus \Omega_m,\overline{D}\setminus \Omega)\to 0\quad \text{ as }\quad m\to \infty,\] 
\cite[Definition 2.2.8]{HENROT-e}. This definition does not depend on the choice of $D$, \cite[Remark 2.2.11]{HENROT-e}. (See \cite{shibahara2021gromov} for a related recent work.)

A sequence $(\Omega_m)_m$ of open sets $\Omega\subset \mathbb{R}^n$ is said to \emph{converge} to an open $\Omega$ \emph{in the sense of characteristic functions} if 
\[\lim_{m\to \infty} \mathbf{1}_{\Omega_m}=\mathbf{1}_{\Omega}\quad \text{in $L^p_{\loc}(\mathbb{R}^n)$ for all $p\in [1,\infty)$},\]
\cite[Definition 2.2.3]{HENROT-e}.

A sequence $(I_m)_m$ of quadratic functionals $I_m:L^2(D)\to [0,+\infty]$ \emph{converges to} a quadratic functional \emph{$I:L^2(D)\to [0,+\infty]$ in the sense of Mosco}, \cite[Definition 2.1.1]{MOSCO94}, if 
\begin{enumerate}
	\item[(i)] we have $\varliminf_{m\to \infty} I_m(u_m)\ge I(u)$ for every seqence $(u_m)_{m}$ converging weakly to $u$ in $L^2(D)$, 
	\item[(ii)] for every $u\in L^2(D)$ there exists a sequence $(u_m)_{m}$ converging strongly in $L^2(D)$ such that $\varlimsup_{m\to \infty} I_m(u_m)\le I(u)$.
\end{enumerate}
A sequence $(I_m)_m$ of quadratic functionals $I_m:L^2(D)\to [0,+\infty]$ \emph{converges to} a quadratic functional \emph{$I:L^2(D)\to [0,+\infty]$ in the Gamma-sense} if (ii) above holds and (i) above holds with weak convergence replaced by strong convergence, see \cite[Section 2]{BucurGiacominiTrebeschi2016}, \cite[Definition 2.2.1]{MOSCO94} or \cite{Braides,DALMASO-1993}. Obviously convergence in the Mosco sense implies convergence in the Gamma-sense.

\begin{theorem}\label{T:Mosco} 
Let $D\subset \mathbb{R}^n$ be a bounded Lipschitz domain, $\alpha\geq 0$ and let $\gamma\in C(\overline{D})$ be strictly positive or identically zero. Let $\varepsilon>0$,  $n-2<d\leq  n$ and let $(\Omega_m)_m$ be a sequence of $(\varepsilon,\infty)$-domains $\Omega_m \subset D$ and $(\mu_m)_m$ a sequence of nonzero Borel measures $\mu_m$, all satisfying \eqref{E:upperdreg} with the same constant and whose supports $\Gamma_m=\supp\mu_m$ are subsets of the $\overline{\Omega}_m$, respectively. For each $m$, let $J(\Omega_m,\mu_m)$ be as in \eqref{extJ--} but with $\Omega_m$, $\mu_m$ in place of $\Omega$, $\mu$. 

If $\lim_m \Omega_m=\Omega$ in the Hausdorff sense and in the sense of characteristic functions and $\lim_m \mu_m=\mu$ weakly, then $\Omega$ is a $(\varepsilon,\infty)$-domain contained in $D$ and $\mu$ is a Borel measure satisfying \eqref{E:upperdreg} and whose support $\Gamma=\supp\mu$ is a subset of $\overline{\Omega}$. Moreover, we have	
\begin{equation}\label{E:Moscolimit}
		\lim_m J(\Omega_m,\mu_m)=J(\Omega,\mu)
\end{equation}
 in the sense of Mosco and in the Gamma-sense.
\end{theorem}

The proof of this result makes use of the following fact shown in \cite[Theorem 2.3]{HINZ-2020}.

\begin{theorem}\label{T:epsinftystable}
	Let $D\subset \mathbb{R}^n$ be a bounded open set and let $\varepsilon>0$. Any sequence $(\Omega_m)_m$ of 
	$(\varepsilon,\infty)$-domains 
	contained in $D$  has a subsequence which converges to an $(\varepsilon,\infty)$-domain $\Omega\subset D$ in the Hausdorff sense. 
\end{theorem}

Theorem \ref{T:Mosco} follows by similar arguments as used for \cite[Theorem 6.3]{HINZ-2020}, for convenience we sketch the necessary modifications.
\begin{proof}
That $\Omega\subset D$ is an $(\varepsilon,\infty)$-domain follows from Theorem \ref{T:epsinftystable}. The support $\Gamma:=\supp\mu$ of $\mu$ is contained in the Hausdorff limit $\lim_m \Gamma_m$, and by \cite[2.2.3.2 and Theorem 2.2.25]{HENROT-e} the latter is a subset of $\overline{\Omega}$.

	Let $(u_m)_m\subset L^2(D)$ be a sequence converging to $u$ weakly in $L^2(D)$ and $(u_{m_k})_k\subset (u_m)_m$ such that $\varliminf_m J(\Omega_m, \mu_m)(u_m)=\lim_k  J(\Omega_{m_k}, \mu_{m_k} )(u_{m_k})$.
	We will show that
	\begin{equation}\label{Eqliminf-}
	\lim_k  J(\Omega_{m_k}, \mu_{m_k})(u_{m_k})\geq J(\Omega,\mu)(u),
	\end{equation}
	what then implies condition (i). We may assume the left hand side of (\ref{Eqliminf-}) is finite, so that $u_{m_k}\in W^{1,2}(\Omega_{m_k})$ for all $k$ and $\sup_k J(\Omega_{m_k}, \mu_{m_k})(u_{m_k})<+\infty$. 
	
	Suppose first that $\alpha>0$ or $\gamma$ is bounded away from zero. Then the preceding implies that $\sup_k \|u_{m_k}\|_{W^{1,2}(\Omega_{m_k})}<+\infty$, in the first case this is immediate, in the second it is due to Corollary \ref{C:equivnorms} and the weak convergence of the measures. By \cite[Theorem 1]{JONES-1981} there exist extensions $E_{\Omega_{m_k}} u_{m_k}\in W^{1,2}(D)$ and a constant $c_{\Ext}>0$ such that $\|E_{\Omega_{m_k}} u_{m_k}\|_{W^{1,2}(D)}\leq c_{\Ext}\|u_{m_k}\|_{W^{1,2}(\Omega_{m_k})}$ for all $k$. Passing to subsequences, we may assume that $(E_{\Omega_{m_k}} u_{m_k})_k$ converges to some $u^\ast$ weakly in $L^2(D)$ and that (by Banach-Saks) the sequence of its convex combinations strongly converges to $u^\ast$. The hypotheses then imply that $u^\ast|_\Omega=u$ and that $\lim_k \mathbf{1}_{\Omega_{m_k}}E_{\Omega_{m_k}} u_{m_k}=\mathbf{1}_\Omega u$ weakly in $L^2(D)$ and hence $\varliminf_k \int_{\Omega_{m_k}}u_{m_k}^2dx\geq \int_\Omega u^2 dx$. We may also assume that $(\nabla E_{\Omega_{m_k}} u_{m_k})_k$ converges 
to some $w^\ast$ weakly in $L^2(D,\mathbb{R}^n)$, with convex combinations converging strongly. Since the convex combinations of 	$(E_{\Omega_{m_k}} u_{m_k})_k$ also converge strongly in $W^{1,2}(D)$, they necessarily converge to $u^\ast$, which therefore is in $W^{1,2}(D)$, and we have $w^\ast=\nabla u^\ast$. Consequently $\lim_k \mathbf{1}_{\Omega_{m_k}}\nabla E_{\Omega_{m_k}} u_{m_k}=\mathbf{1}_\Omega \nabla u$ weakly in $L^2(D,\mathbb{R}^n)$ and therefore
\begin{equation}\label{E:gradientsconv}
\varliminf_k \int_{\Omega_{m_k}} |\nabla u_{m_k}|^2dx\geq \int_\Omega |\nabla u|^2 dx.
\end{equation}
Combining and using the superadditivity of $\varliminf$, we arrive at
\begin{equation}\label{E:insidedomain}
\varliminf_k \left\lbrace \int_{\Omega_{m_k}} |\nabla u_{m_k}|^2dx + \alpha \int_{\Omega_{m_k}}u_{m_k}^2dx \right\rbrace \geq \int_\Omega |\nabla u|^2 dx + \alpha \int_\Omega u^2 dx.
\end{equation}
Using Lemma \ref{L:traceconvergence} and proceeding as in the proof of \cite[Theorem 6.3]{HINZ-2020} we can obtain the limit relation
\[\lim_k \int_{\Gamma_{m_k}} (\mathrm{Tr}_{\Omega_{m_k},\Gamma_{m_k}} u_{m_k})^2 d\mu_{m_k}=\int_\Gamma (\mathrm{Tr}_{\Omega,\Gamma}u)^2 d\mu,\]
and together with (\ref{E:insidedomain}) this shows (\ref{Eqliminf-}) and therefore (i).
	
In the case that $\alpha=0$ and $\gamma\equiv 0$ we have $\sup_k \|\nabla u_{m_k}\|_{L^2(\Omega_{m_k})}<+\infty$ and there are 
extensions $E_{\Omega_{m_k}} u_{m_k}\in L^0(D)$ and a constant $c_{\Ext}>0$ such that $\|\nabla E_{\Omega_{m_k}} u_{m_k}\|_{L^2(D)}\leq c_{\Ext}\|\nabla u_{m_k}\|_{L^2(\Omega_{m_k})}$ for all $k$, \cite[Theorem 2]{JONES-1981}. By Poincar\'e's inequality for $D$ we may assume that the sequence $((E_{\Omega_{m_k}}u_{m_k}-c_k))_k$, where $c_k=\lambda^n(D)^{-1}\int_{D}E_{\Omega_{m_k}}u_{m_k} dx$, converges to some $u^\ast$ weakly in $L^2(D)$, and the sequence of its convex combinations converges strongly. Together with the hypothesis this shows that $\mathbf{1}_\Omega u -\mathbf{1}_\Omega u^\ast=\lim_k \mathbf{1}_{\Omega_{m_k}}c_k$ must equal $c^\ast \mathbf{1}_\Omega$ with a constant $c^\ast \in\mathbb{R}$. On the other hand we may also assume that $(\nabla E_{\Omega_{m_k}}u_{m_k})_k$ converges to some $w^\ast$ weakly in $L^2(D,\mathbb{R}^n)$, with its convex combinations converging strongly. Each single function $E_{\Omega_{m_k}}u_{m_k}$ is in $W^{1,2}(D)$, and the image of $W^{1,2}(D)$ under $\nabla$ is a closed subspace of $L^2(D,\mathbb{R}^n)$, see for instance \cite[Chapter XIX, Section 1.3, Theorem 4]{Dutray-Lions-V6}. Since it is then also weakly closed, we must have $w^\ast =\nabla g$ with some $g\in W^{1,2}(D)$, and without loss of generality $\int_\Omega g\:dx=0$. Since the convex combinations of the gradients $\nabla E_{\Omega_{m_k}}u_{m_k}$ converge to $\nabla g$ strongly in $L^2(D,\mathbb{R}^n)$, another application of Poincar\'e's inequality shows that $g=u^\ast$. Using the convergence in the sense of characteristic functions we can again conclude that $\lim_k \mathbf{1}_{\Omega_{m_k}}\nabla E_{\Omega_{m_k}} u_{m_k}=\mathbf{1}_\Omega \nabla u$ weakly in $L^2(D,\mathbb{R}^n)$, so that (\ref{E:gradientsconv}) holds, which in the present case shows (\ref{Eqliminf-}) and therefore (i).

Condition (ii) can be verified as in \cite[Theorem 6.3]{HINZ-2020}.
\end{proof}

We record the following simple observation.
\begin{proposition}\label{P:projectorsconv}
If $D\subset \mathbb{R}^n$ is a bounded domain, $\Omega_m,\Omega\subset D$ are subdomains and $\Omega_m\to \Omega$ in the sense of characteristic functions, then $\lim_{m\to \infty} P_{\Omega_m}=P_\Omega$ in the strong sense on $ L^2(D)$.
\end{proposition}

\begin{proof}
Given $\psi\in L^2(D)$ we have $\left\|P_{\Omega_m}\psi-P_{\Omega}\psi\right\|_{L^2(D)}^2=\int_D|\mathbf{1}_{\Omega_m}-\mathbf{1}_\Omega|^2\psi^2dx=0$. This goes to zero by Vitali's convergence theorem, note that the sequence $(\mathbf{1}_{\Omega_m})_m$ is $L^2$-uniformly integrable w.r.t. the finite measure $\psi^2\lambda^n$, and the convergence of the domains in the sense of characteristic functions implies that $\mathbf{1}_{\Omega_m}\to \mathbf{1}_\Omega$ in $\lambda^n$-measure and by absolute continuity then also in $\psi^2\lambda^n$-measure. 
\end{proof}

The following preliminary convergence results for the resolvents (and quasi-inverses) is closely related to \cite[Theorem 3.1]{Weidmann1984}, see also \cite[p. 155/156]{Weidmann97}. 

\begin{corollary}\label{C:weaksolconv}
Let $n$, $D$, $\alpha$, $\gamma$, $\varepsilon$, $d$, $(\Omega_m)_m$ and $(\mu_m)_m$ be as in Theorem \ref{T:Mosco}. Suppose that $\lim_m \Omega_m=\Omega$ in the Hausdorff sense and in the sense of characteristic functions and $\lim_m \mu_m=\mu$ weakly. 
For any $\alpha>0$ we have both $\lim_{m\to\infty} \hat{G}_{\alpha,\gamma}^{\Omega_m,\mu_m,\ast}=\hat{G}_{\alpha, \gamma}^{\Omega,\mu,\ast}$ and 
\begin{equation}\label{E:quasiinvconv}
\lim_{m\to\infty}P_{\Omega_m}\circ \hat{G}_{\alpha,\gamma}^{\Omega_m,\mu_m,\ast}=P_{\Omega}\circ \hat{G}_{\alpha,\gamma}^{\Omega,\mu,\ast},
\end{equation}
in the strong sense on $L^2(D)$. If $\gamma\equiv 0$, $\alpha=0$ and $\ast=D$ or $\gamma$ is bounded away from zero, $\alpha=0$ and $\ast=R$, then (\ref{E:quasiinvconv}) still holds.
\end{corollary}

We write $\left\|\cdot\right\|$ to denote the operator norm for operators from $L^2(D)$ to itself.

\begin{proof}
Suppose $\alpha>0$. Then the strong convergence of the resolvents is immediate from Theorem \ref{T:Mosco} and \cite[Theorem 2.4.1]{MOSCO94}. Identity (\ref{E:quasiinvconv}) follows using
\begin{align}
\Big\|P_{\Omega_m}\circ \hat{G}_{\alpha,\gamma}^{\Omega_m,\mu_m,\ast}f & -P_{\Omega}\circ \hat{G}_{\alpha,\gamma}^{\Omega,\mu,\ast}f\Big\|_{L^2(D)}\notag\\
& \leq\left\|P_{\Omega_m}\circ \hat{G}_{\alpha,\gamma}^{\Omega_m,\mu_m,\ast}f-P_{\Omega_m}\circ \hat{G}_{\alpha,\gamma}^{\Omega,\mu,\ast}f\right\|_{L^2(D)}
+\left\|(P_{\Omega_m}-P_\Omega)\circ \hat{G}_{\alpha,\gamma}^{\Omega,\mu,\ast}f\right\|_{L^2(D)},\notag
\end{align}
valid for all $f\in L^2(D)$, together with Proposition \ref{P:projectorsconv}. A key ingredient for this argument is the trivial uniform bound on the resolvents,
$\sup_m\big\|\hat{G}_{\alpha,\gamma}^{\Omega_{m},\mu_{m},\ast}\big\|\leq \frac{1}{\alpha}$. Under hypotheses stated for $\alpha=0$ we still have
\begin{equation}\label{E:supbound}
\sup_m \big\|P_{\Omega_m}\circ \hat{G}_{0,\gamma}^{\Omega_m,\mu_m,\ast}\big\| <+\infty, 
\end{equation}
note that by Theorem \ref{T:uniPoincare} and Corollary \ref{C:equivnorms}, combined with (\ref{E:GreenD}) respectively (\ref{E:GreenR}), we have 
\[\big\|\hat{G}_{0,\gamma}^{\Omega_m,\mu_m,\ast}f\big\|_{L^2(\Omega_m)}^2\leq c\:\mathcal{E}_{0,\gamma}^{\Omega_m,\mu_m,\ast}(\hat{G}_{0,\gamma}^{\Omega_m,\mu_m,\ast}f)\leq c\:\left\|f\right\|_{L^2(D)}\big\|\hat{G}_{0,\gamma}^{\Omega_m,\mu_m,\ast}f\big\|_{L^2(\Omega_m)}, \quad f\in L^2(D),\]
with a constant $c>0$ independent of $m$. To discuss the case $\alpha=0$ under the stated hypotheses, note that Remark \ref{R:quasiinv} (ii) we have $P_{\Omega}\circ \hat{G}_{\beta,\gamma}^{\Omega,\mu,\ast}\circ \hat{G}_{\alpha,\gamma}^{\Omega,\mu,\ast}\circ P_{\Omega}
=P_{\Omega}\circ \hat{G}_{\beta,\gamma}^{\Omega,\mu,\ast}\circ P_\Omega\circ \hat{G}_{\alpha,\gamma}^{\Omega,\mu,\ast}\circ P_{\Omega}$, and similarly for $\Omega_m$ and $\mu_m$ in place of $\Omega$ and $\mu$. Therefore the operators 
\[P_{\Omega}\circ \hat{G}_{\alpha,\gamma}^{\Omega,\mu,\ast}\circ P_{\Omega},\quad  \alpha\geq 0,\quad  \text{and}\quad  P_{\Omega_m}\circ \hat{G}_{\alpha,\gamma}^{\Omega_m,\mu_m,\ast}\circ P_{\Omega_m},\quad \alpha\geq 0,\] 
are quickly seen to satisfy the resolvent equation. As a consequence, 
\begin{multline}
P_{\Omega_m}\circ \hat{G}_{\alpha,\gamma}^{\Omega_m,\mu_m,\ast}\circ P_{\Omega_m}-P_{\Omega}\circ \hat{G}_{\alpha,\gamma}^{\Omega,\mu,\ast}\circ P_{\Omega}\notag\\
=\left(1+(\alpha-\beta)P_{\Omega_m}\circ \hat{G}_{\alpha,\gamma}^{\Omega_m,\mu_m,\ast}\circ P_{\Omega_m}\right)\left(P_{\Omega_m}\circ \hat{G}_{\beta,\gamma}^{\Omega_m,\mu_m,\ast}\circ P_{\Omega_m}-P_{\Omega}\circ \hat{G}_{\beta,\gamma}^{\Omega,\mu,\ast}\circ P_{\Omega}\right)\times\notag\\
\times\left(1+(\alpha-\beta)P_{\Omega}\circ \hat{G}_{\alpha,\gamma}^{\Omega, \mu,\ast}\circ P_{\Omega}\right)
\end{multline}
for all $\alpha,\beta\geq 0$, and combining this with the uniform bound (\ref{E:supbound}) we can follow the proof of \cite[Theorem VIII.1.3]{Kato-1980} to see that (\ref{E:quasiinvconv}) actually remains valid for $\alpha=0$.
\end{proof}

\begin{remark}\mbox{}
\begin{enumerate}
\item[(i)] In Theorem \ref{T:normresconv} below we upgrade this convergence statement to convergence in operator norm. The convergence for the case $\alpha=0$ under the stated hypotheses could then also be concluded from \cite[Chapter IV, Section 2.6, Theorem 2.25]{Kato-1980}.
\item[(ii)] Using the methods in \cite{Post-2012} one could also bypass Mosco convergence and provide a direct proof of a generalized form of norm resolvent convergence. This would easily allow to include the convergence in the case $\ast=N$ and $\alpha=0$ and to investigate rates of convergence. Since the focus here is different, we leave it for a follow up paper.
\end{enumerate}
\end{remark}

To simplify the statement of further results on the convergence of resolvents, we formulate an assumption.
\begin{assumption}\label{A:alphagamma}
The parameters $\alpha$, $\gamma$ and $\ast$ satisfy one of the following:
\begin{itemize}
\item $\gamma\equiv 0$, $\alpha\geq 0$ and $\ast=D$, or 
\item $\gamma$ is bounded away from zero, $\alpha\geq 0$ and $\ast=R$, or
\item $\gamma\equiv 0$,  $\alpha>0$ and $\ast=N$.
\end{itemize}
\end{assumption}

The following theorem may be seen as a 'joint' upgrade of Lemma \ref{L:compactnessprop} and Corollary \ref{C:weaksolconv}. The first statement is similar to \cite[Lemma 4.8]{BucurGiacominiTrebeschi2016}. 

\begin{theorem}\label{T:compactnesspropvar}
Let $n$, $D$, $\alpha$, $\gamma$, $\varepsilon$, $d$, $(\Omega_m)_m$ and $(\mu_m)_m$ be as in Theorem \ref{T:Mosco} and let Assumption \ref{A:alphagamma} be satisfied.
Suppose that $\lim_m \Omega_m=\Omega$ in the Hausdorff sense and in the sense of characteristic functions and $\lim_m \mu_m=\mu$ weakly. Let $(g_m)_{m=1}^\infty\subset L^2(D)$ be a sequence that converges weakly in $L^2(D)$ to some $g$. Then there is a sequence $(m_k)_{k=1}^\infty$ with $m_k\uparrow +\infty$ such that the following hold:
\begin{enumerate}
\item[(i)] We have 
\[\lim_{k\to \infty} P_{\Omega_{m_k}}\circ \hat{G}_{\alpha,\gamma}^{\Omega_{m_k},\mu_{m_k},\ast}g_{m_k}=P_\Omega\circ \hat{G}_{\alpha,\gamma}^{\Omega,\mu,\ast}g\quad \text{in $L^2(D)$}.\]
\item[(ii)] If $u_m$ and $u$ be the unique weak solutions to (\ref{E:abstractDirichlet}) or (\ref{E:abstractRobin}) with zero boundary data on $\Omega_m$ with $f=g_m$ and on $\Omega$ with $f=g$, then there is some $u^\ast\in W^{1,2}(D)$ with $u^\ast|_\Omega=u$ such that 
\begin{equation}\label{E:weakW12}
\lim_{k\to \infty} E_{\Omega_{m_k}}u_{m_k}=u^\ast\quad \text{weakly in $W^{1,2}(D)$.}
\end{equation}
Moreover, we have 
\begin{equation}\label{E:limitenergy}
\lim_{k\to \infty} J(\Omega_{m_k},\mu_{m_k})(u_{m_k})=J(\Omega,\mu)(u)
\end{equation}
and 
\begin{equation}\label{E:strongW12}
\lim_{k\to \infty} E_{\Omega_{m_k}}u_{m_k}=u\quad \text{in $W^{1,2}(\Omega)$.}
\end{equation}
\end{enumerate}
\end{theorem}

\begin{remark}\label{R:compactnesspropvar}
Corollary \ref{C:weaksolconv} and Theorem \ref{T:compactnesspropvar} (applied with $g=g_m=f$) constitute a stability result for the respective boundary value problem along a sequence of varying domains in a common confinement.
\end{remark}

\begin{proof}
Let $v_m:=P_{\Omega_m}\circ \hat{G}_{\alpha,\gamma}^{\Omega_m,\mu_m,\ast}g_m$ and $v:=P_\Omega \circ \hat{G}_{\alpha,\gamma}^{\Omega,\mu,\ast}g$ and recall $v_m$ and $v$  the continuations by zero of the unique weak solutions $u_m$ and $u$ to (\ref{E:abstractDirichlet}) or (\ref{E:abstractRobin}) as stated. We have
\begin{equation}\label{E:weakconvvm}
\lim_{m\to \infty} v_m=v\quad \text{ weakly in $L^2(D)$},
\end{equation}
note that 
\begin{multline}
\lim_{m\to \infty}\left\langle v_m,w\right\rangle_{L^2(D)}=\lim_{m\to \infty}\left\langle g_m, P_{\Omega_{m}}\circ \hat{G}_{\alpha,\gamma}^{\Omega_{m},\mu_{m},\ast}P_{\Omega_{m}}w\right\rangle_{L^2(D)}\notag\\
=\left\langle g, P_{\Omega}\circ\hat{G}_{\alpha,\gamma}^{\Omega,\mu,\ast}P_{\Omega}w\right\rangle_{L^2(D)}=\left\langle v,w\right\rangle_{L^2(D)}
\end{multline}
for any $w\in L^2(D)$ by Remark \ref{R:quasiinv} (ii) and Corollary \ref{C:weaksolconv}.

Testing the solution of the respective problem against itself, we observe that
\[\left\|v_m\right\|_{W^{1,2}(\Omega_m)}^2\leq c\:\mathcal{E}_{\alpha,\gamma}^{\Omega_m,\mu_m}(v_m)=c\:\int_{\Omega_m}g_mv_m\:dx\leq c\:\left(\sup_m \left\|g_m\right\|_{L^2(D)}\right)\left\|v_m\right\|_{L^2(\Omega_m)}\]
with a constant $c>0$ independent of $m$. The independence of $c$ of $m$ is obvious if $\alpha>0$, it follows from Corollary \ref{C:equivnorms} if $\alpha=0$ and $\ast =D$ and or $\ast=R$. Consequently 
\[\sup_m\left\|v_m\right\|_{W^{1,2}(\Omega_m)}<+\infty.\]

Since all $\Omega_m$ are $(\varepsilon,\infty)$-domains with the same $\varepsilon$, we can use \cite[Theorems 1]{JONES-1981} to see that there are extensions $E_{\Omega_m}(v_m|_{\Omega_m})\in W^{1,2}(D)$ of $u_m=v_m|_{\Omega_m}$ such that 
\begin{equation}\label{E:boundedonD}
\sup_m\left\|E_{\Omega_m}(v_m|_{\Omega_m})\right\|_{W^{1,2}(D)}<+\infty.
\end{equation}

As a consequence there are a sequence $(m_k)_k$ with $m_k\uparrow +\infty$ and a function $u^\ast\in W^{1,2}(D)$ such that (\ref{E:weakW12}) holds, and by Banach-Saks we may assume the convex combinations of the functions $E_{\Omega_{m_k}}u_{m_k}$ converge to $u^\ast$ in $W^{1,2}(D)$. By Rellich-Kondrachov we may assume that also $\lim_{k\to \infty}E_{\Omega_{m_k}}u_{m_k}=w$ in $L^2(D)$ with some suitable $w\in L^2(D)$, and by the preceding we then must have $w=u^\ast$. Using Proposition \ref{P:projectorsconv} it follows that 
\[\lim_{k\to \infty} v_{m_k}=\lim_{k\to \infty} P_{\Omega_{m_k}}v_{m_k}=\lim_{k\to \infty} P_{\Omega_{m_k}} E_{\Omega_{m_k}}u_{m_k}=P_\Omega u^\ast\] 
in $L^2(D)$, so that $P_\Omega u^\ast=v$ by (\ref{E:weakconvvm}). This proves (i) and (\ref{E:weakW12}). 

To prove (\ref{E:limitenergy}) we can now proceed as in \cite[Theorem 5.3]{CAPITANELLI-2010-1}. By (i) we have 
\begin{equation}\label{E:inhomoconv}
\lim_{k\to \infty}\int_{\Omega_{m_k}}u_{m_k}g_{m_k}\:dx=\int_\Omega u g\:dx.
\end{equation}
The first condition of the convergence in the sense of Mosco  implies that 
\begin{equation}\label{E:harvestfirst}
\frac12 J(\Omega,\mu)(u)-\int_\Omega u g\:dx\leq \varliminf_{k\to \infty}\left\lbrace \frac12 J(\Omega_{m_k},\mu_{m_k})(u_{m_k})-\int_{\Omega_{m_k}}u_{m_k} g_{m_k}\:dx\right\rbrace.
\end{equation}
By the second condition there is a sequence $(w_k)_k\subset L^2(D)$ such that $\lim_{k\to \infty} w_k=u^\ast$ in $L^2(D)$ and 
\begin{equation}\label{E:harvestsecond}
\varlimsup_{k\to \infty} \left\lbrace \frac12 J(\Omega_{m_k},\mu_{m_k})(w_k)-\int_{\Omega_{m_k}}w_{k} g_{m_k}\:dx\right\rbrace\leq \frac12 J(\Omega,\mu)(u)-\int_\Omega u g\:dx.
\end{equation}
Since by Proposition \ref{P:weaksol} the function $u_{m_k}$ minimizes $w\mapsto \frac12 J(\Omega_{m_k},\mu_{m_k})(w)-\int_{\Omega_{m_k}}w g_{m_k}\:dx$, it follows that (\ref{E:harvestsecond}) dominates 
\[\varlimsup_{k\to \infty} \left\lbrace \frac12 J(\Omega_{m_k},\mu_{m_k})(u_{m_k})-\int_{\Omega_{m_k}}u_{m_k} g_{m_k}\:dx\right\rbrace.\]
Combining with (\ref{E:harvestfirst}) and taking into account (\ref{E:inhomoconv}), we obtain (\ref{E:limitenergy}). 

To see (\ref{E:strongW12}), note that by (i) we have $\lim_{k\to \infty} \int_{\Omega_{m_k}}u_{m_k}^2\:dx=\int_\Omega u^2\:dx$, and since $\lim_{m\to \infty} \gamma\cdot \mu_m=\gamma\cdot \mu$ weakly, Lemma \ref{L:traceconvergence} and (\ref{E:weakW12}) imply that 
\[\lim_{k\to \infty} \int_{\Gamma_{m_k}}\gamma (\mathrm{Tr}_{\Omega_{m_k},\Gamma_{m_k}} u_{m_k})^2 \:d\mu_{m_k}=\int_{\Gamma}\gamma (\mathrm{Tr}_{\Omega,\Gamma} u)^2 \:d\mu.\]
Together with (\ref{E:limitenergy}) it follows that 
\begin{equation}\label{E:movingforms}
\lim_{k\to \infty}\left\|u_{m_k}\right\|_{W^{1,2}(\Omega_{m_k})}=\left\|u \right\|_{W^{1,2}(\Omega)}.
\end{equation}
By the above and by the Helmholtz decomposition on $D$, \cite[Chapter XIX, Section 1.3, Theorem 4]{Dutray-Lions-V6}, the sequence $(\nabla E_{\Omega_{m_k}}u_{m_k})_k$ converges to $\nabla u^\ast$ weakly in $L^2(D,\mathbb{R}^n)$, and by the convergence in the sense of characteristic functions we obtain
\[\lim_{k\to \infty} \int_D (\mathbf{1}_{\Omega_{m_k}}-\mathbf{1}_\Omega) (\nabla E_{\Omega_{m_k}}u_{m_k})^2\:dx=0.\]
Similarly, we have 
\[\lim_{k\to \infty} \int_D (\mathbf{1}_{\Omega_{m_k}}-\mathbf{1}_\Omega) (E_{\Omega_{m_k}}u_{m_k})^2\:dx=0,\]
and combining with (\ref{E:movingforms}), it follows that 
\[\lim_{k\to \infty}\big\|E_{\Omega_{m_k}} u_{m_k}\big\|_{W^{1,2}(\Omega)}=\left\|u \right\|_{W^{1,2}(\Omega)}.\]
Since by (\ref{E:boundedonD}) we may assume that $\lim_{k\to \infty} E_{\Omega_{m_k}} u_{m_k}=u$ weakly in $W^{1,2}(\Omega)$, we arrive at (\ref{E:strongW12}).
\end{proof}

\section{Norm resolvent convergence and eigenvalues}\label{S:eigenvalues}

As before let $D\subset \mathbb{R}^n$, $n\geq 2$, and suppose that $\Omega\subset D$ is a $W^{1,2}$-extension domain. Let $\gamma:\overline{D}\to [0,+\infty)$ be a bounded Borel function. Assume that $\gamma\equiv 0$ and $\ast=D$, that $\gamma$ is bounded away from zero and $\ast=R$, or that $\gamma\equiv 0$ and $\ast=N$. Recall that $\mathcal{L}^{\Omega,\mu,\ast}_\gamma$ denotes the generator of $\mathcal{E}_{0,\gamma}^{\Omega,\mu}$, Remark \ref{R:quasiinv} (i). The eigenvalues 
\[0\leq \lambda_1(\Omega,\mu,\ast)\leq \lambda_2(\Omega,\mu,\ast)\leq ...\]
of $-\mathcal{L}^{\Omega,\mu,\ast}_\gamma$ have finite multiplicity and accumulate only at infinity, and the eigenfunctions $\varphi_n(\Omega,\mu,\ast)$ corresponding to the eigenvalues $\lambda_n(\Omega,\mu,\ast)$, respectively, form a complete orthonormal system $(\varphi_n(\Omega,\mu,\ast))_n$ of $L^2(\Omega)$. Given $a<b$, let $\pi_{(a,b)}(\Omega,\mu,\ast)$ denote the
spectral projector associated with the interval $(a,b)$.

We denote the continuation of $\varphi_n(\Omega,\mu,\ast)$ by zero to $D$ by the same symbol. Let $(\chi_n(D\setminus \Omega,\mu,\ast))_n$ be a complete orthonormal system in $L^2(D\setminus \Omega)$ and denote the extensions by zero of its elements again  by the same symbols. Then $(\varphi_n(\Omega,\mu,\ast))_n\cup (\chi_n(D\setminus \Omega,\mu,\ast))_n$ is a complete orthonormal system in $L^2(D)$, the $\lambda_n(\Omega,\mu,\ast)$ are also eigenvalues of $\hat{\mathcal{E}}_{0,\gamma}^{\Omega,\mu}$ on $L^2(D)$, in the sense that we have $\hat{\mathcal{E}}_{0,\gamma}^{\Omega,\mu}(\varphi_n(\Omega,\mu,\ast),\psi)=\lambda_n(\Omega,\mu,\ast)\int_D \varphi_n(\Omega,\mu,\ast)\:\psi\:dx$
for all $\psi$ from $V(\Omega,\Gamma)_D$ respectively $W^{1,2}(\Omega)_D$, and all $\chi_n(D\setminus \Omega,\mu,\ast)$
are eigenfunctions corresponding to the eigenvalue zero.

Now suppose we are in the situation of Theorem \ref{T:Mosco}. Theorem \ref{T:compactnesspropvar} allows to upgrade from the strong convergence of the resolvents observed in Corollary \ref{C:weaksolconv} and Theorem \ref{T:compactnesspropvar} to norm convergence; this is similar to \cite[Theorem 3.2]{BucurGiacominiTrebeschi2016}. We can then conclude convergence properties for spectral projectors, eigenvalues and eigenfunctions.

\begin{theorem}\label{T:normresconv}
Let $n$, $D$, $\alpha$, $\gamma$, $\varepsilon$, $d$, $(\Omega_m)_m$ and $(\mu_m)_m$ be as in Theorem \ref{T:Mosco}.
Suppose that $\lim_m \Omega_m=\Omega$ in the Hausdorff sense and in the sense of characteristic functions and $\lim_m \mu_m=\mu$ weakly. There is a sequence $(m_k)_{k=1}^\infty$ with $m_k\uparrow +\infty$ such that the following hold.
\begin{enumerate}
\item[(i)] If Assumption \ref{A:alphagamma} holds is satisfied, then 
\[\lim_{k\to \infty} P_{\Omega_{m_k}}\circ \hat{G}_{\alpha,\gamma}^{\Omega_{m_k},\mu_{m_k},\ast}=P_\Omega\circ \hat{G}_{\alpha,\gamma}^{\Omega,\mu,\ast}\]
in operator norm. 
\item[(ii)] Let
$\lambda_+(\Omega,\mu,\ast)$ denote the smallest nonzero eigenvalue of $-\mathcal{L}^{\Omega,\mu,\ast}_\gamma$, let $a$ and $b$ be elements of the resolvent set of $-\mathcal{L}^{\Omega,\mu,\ast}_\gamma$ with $0<a<b$. Then we have 
\[\lim_{k\to \infty} \pi_{(a,b)}(\Omega_{m_k},\mu_{m_k},\ast)=\pi_{(a,b)}(\Omega,\mu,\ast)\] 
in operator norm.
\item[(iii)] The spectra of the operators $-\mathcal{L}^{\Omega_{m_k},\mu_{m_k},\ast}_\gamma$ converge to the spectrum of $-\mathcal{L}^{\Omega,\mu,\ast}_\gamma$ in the Hausdorff sense. The eigenvalues $\lambda_n(\Omega,\mu,\ast)$ of the operator $-\mathcal{L}^{\Omega,\mu,\ast}_\gamma$ are exactly the limits as $k\to \infty$ of sequences of the eigenvalues of the operators $-\mathcal{L}^{\Omega_{m_k},\mu_{m_k},\ast}_\gamma$,
\begin{equation}\label{E:eivalconv}
\lambda_n(\Omega,\mu,\ast)=\lim_{k\to \infty}\lambda_n(\Omega_{m_k},\mu_{m_k},\ast).
\end{equation}
\item[(iv)] For each $n$ we can find a sequence of normalized eigenfunctions $\varphi^{(m_k)}_n$ of the operators $-\mathcal{L}^{\Omega_{m_k},\mu_{m_k},\ast}_\gamma$ that converges to $\varphi_n(\Omega,\mu,\ast)$ in $L^2(D)$. Moreover, for any $\Phi\in C(\mathbb{R})$ with $\Phi(0)=0$ we have $\lim_{k\to \infty}\Phi(\mathcal{L}^{\Omega_{m_k},\mu_{m_k},\ast}_\gamma)\varphi^{(m_k)}_n=
\Phi(\mathcal{L}^{\Omega,\mu,\ast}_\gamma)\varphi_n(\Omega,\mu,\ast)$ in $L^2(D)$ (in the sense of continuation by zero). If all eigenvalues $\lambda_n(\Omega_{m_k},\mu_{m_k},\ast)$ are simple, then the $\varphi^{(m_k)}_n$ equal $\varphi_n(\Omega_{m_k},\mu_{m_k},\ast)$, $k\geq 1$, up to signs.

For any $n$ the sequence $(\varphi_n(\Omega_{m_k},\mu_{m_k},\ast))_k$ has a subsequence converging  in $L^2(D)$ to a normalized eigenfunction $\varphi_n$ of $-\mathcal{L}^{\Omega,\mu,\ast}_\gamma$ corresponding to $\lambda_n(\Omega,\mu,\ast)$. For $\Phi$ as specified above also the respective images of the members of this subsequence under the operators $\Phi(\mathcal{L}^{\Omega_{m_k},\mu_{m_k},\ast}_\gamma)$ converge to $\Phi(\mathcal{L}^{\Omega,\mu,\ast}_\gamma)\varphi_n$ in $L^2(D)$. If $\lambda_n(\Omega,\mu,\ast)$ is a simple eigenvalue, then the full sequence converges to $\varphi_n=\pm\varphi_n(\Omega,\mu,\ast)$.
\end{enumerate}
\end{theorem}

\begin{proof}
To see (i) we can follow the proof of \cite[Theorem 3.2]{BucurGiacominiTrebeschi2016}. It suffices to verify that along a sequence $(m_k)_k$ as stated, we have 
\begin{equation}\label{E:provenormresconv}
\lim_{k\to \infty} \sup_{g\in L^2(D),\ \left\|g\right\|_{L^2(D)}\leq 1}\left\|P_{\Omega_{m_k}}\circ \hat{G}_{\alpha,\gamma}^{\Omega_{m_k},\mu_{m_k},\ast}g-P_{\Omega}\circ \hat{G}_{\alpha,\gamma}^{\Omega,\mu,\ast}g\right\|_{L^2(D)}=0.
\end{equation}
For each $m$ we can, by Lemma \ref{L:compactnessprop}, find a function $g_m\in L^2(D)$ with $\left\|g_m\right\|_{L^2(D)}\leq 1$
such that 
\begin{multline}\label{E:minimizer}
\left\|P_{\Omega_m}\circ \hat{G}_{\alpha,\gamma}^{\Omega_{m},\mu_{m},\ast}g_m-P_\Omega\circ \hat{G}_{\alpha,\gamma}^{\Omega,\mu,\ast}g_m\right\|_{L^2(D)}\\
=\sup_{g\in L^2(D),\ \left\|g\right\|_{L^2(D)}\leq 1}\left\|P_{\Omega_{m}}\circ \hat{G}_{\alpha,\gamma}^{\Omega_{m},\mu_{m},\ast}g-P_{\Omega}\circ \hat{G}_{\alpha,\gamma}^{\Omega,\mu,\ast}g\right\|_{L^2(D)}.
\end{multline}
Since $(g_m)_m\subset L^2(D)$ is bounded, we can find a sequence $(m_k)_k$ with $m_k\uparrow +\infty$ such that $\lim_{k\to \infty} g_{m_k}=g$ weakly in $L^2(D)$. Another application of Lemma \ref{L:compactnessprop} gives 
\[\lim_{k\to \infty} \left\|P_\Omega\circ \hat{G}_{\alpha,\gamma}^{\Omega,\mu,\ast}g_{m_k}-P_\Omega\circ \hat{G}_{\alpha,\gamma}^{\Omega,\mu,\ast}g\right\|_{L^2(D)}=0,\]
and from Theorem \ref{T:compactnesspropvar} we obtain 
\[\lim_{k\to \infty} \left\|P_{\Omega_{m_k}}\circ \hat{G}_{\alpha,\gamma}^{\Omega_{m_k},\mu_{m_k},\ast}g_{m_k}-P_\Omega\circ \hat{G}_{\alpha,\gamma}^{\Omega,\mu,\ast}g\right\|_{L^2(D)}=0.\]
Combining, using the triangle inequality and taking into account (\ref{E:minimizer}), we arrive at (\ref{E:provenormresconv}).

The projection $P_\Omega$ does not affect the complement of the eigenspace of $-\mathcal{L}^{\Omega,\mu,\ast}_\gamma$ for the eigenvalue zero. Therefore statement (ii) follows using well-known arguments, \cite[Theorem VIII.23]{REED-SIMON-1980}. See also \cite[p. 154 and the related Theorem 2]{Weidmann97}. 

The first part in (iii) follows from \cite[Chapter IV, Section 3.1, Theorem 3.1 and Remark 3.3]{Kato-1980}, applied to the resolvents, together with \cite[Chapter VIII, Section 1.2, Theorem 1.14]{Kato-1980}. 

Next, note that the first eigenvalue 
$\lambda_1(\Omega,\mu,\ast)$ of $-\mathcal{L}^{\Omega,\mu,\ast}_\gamma$ is zero if and only $\ast=N$ and $\gamma\equiv 0$. In this case also $\lambda_1(\Omega_{m_k},\mu_{m_k},\ast)=0$ for all $k$. The corresponding normalized eigenfunctions are $\varphi_1(\Omega,\mu,N)=(\lambda^n(\Omega))^{-1}\mathbf{1}_{\Omega}$ and $\varphi_1(\Omega_{m_k},\mu_{m_k},N)=(\lambda^n(\Omega_{m_k}))^{-1}\mathbf{1}_{\Omega_{m_k}}$, and by the convergence of the domains in the sense of characteristic functions (and the bounded convergence theorem) we have $\lim_{k\to \infty} \varphi_1(\Omega_{m_k},\mu_{m_k},N)=\varphi_1(\Omega,\mu,N)$ in $L^2(D)$.

Now consider the smallest nonzero eigenvalue $\lambda_+(\Omega,\mu,\ast)$ and let $(a,b)$, with $0<a<b$ from the resolvent set of $\mathcal{L}^{\Omega,\mu,\ast}_\gamma$, be a small interval around it and containing no other eigenvalue of $-\mathcal{L}^{\Omega,\mu,\ast}_\gamma$. Then for large enough $k$ the number $a$ is in the resolvent set of $-\mathcal{L}^{\Omega_{m_k},\mu_{m_k},\ast}_\gamma$, the interval $(a,b)$ also contains its smallest nonzero eigenvalue $\lambda_+(\Omega_{m_k},\mu_{m_k},\ast)$, but no other of its eigenvalues; moreover, $\lambda_+(\Omega_{m_k},\mu_{m_k},\ast)$ and $\lambda_+(\Omega,\mu,\ast)$ have the same multiplicity, \cite[Chapter IV, Section 3.4, Theorem 3.16, and Section 3.5]{Kato-1980}. We have 
\[\big| \lambda_+(\Omega_{m_k},\mu_{m_k},\ast) -\lambda_+(\Omega,\mu,\ast)\big|\leq (b-a)^2\left|\frac{1}{\lambda_+(\Omega_{m_k},\mu_{m_k},\ast)-a}-\frac{1}{\lambda_+(\Omega,\mu,\ast)-a}\right|.\]
By the preceding $\pi_{(a,b)}(\Omega,\mu,\ast)=P_\Omega\circ\pi_{(a,b)}(\Omega,\mu,\ast)$ is the orthogonal projection onto the eigenspace for $\lambda_+(\Omega,\mu,\ast)$, we therefore have 
\[\big\|\hat{G}_{a,\gamma}^{\Omega,\mu,\ast}\circ \pi_{(a,b)}(\Omega,\mu,\ast)\big\|=\frac{1}{\lambda_+(\Omega,\mu,\ast)-a},\]
and similarly for $\Omega_{m_k}$ and $\mu_{m_k}$ in place of $\Omega$ and $\mu$. Using the reverse triangle inequality, we therefore have 
\begin{multline}\left|\frac{1}{\lambda_+(\Omega_{m_k},\mu_{m_k},\ast)-a}-\frac{1}{\lambda_+(\Omega,\mu,\ast)-a}\right|\notag\\
\leq \big\|\hat{G}_{a,\gamma}^{\Omega_{m_k},\mu_{m_k},\ast}\circ \pi_{(a,b)}(\Omega_{m_k},\mu_{m_k},\ast)-\hat{G}_{a,\gamma}^{\Omega,\mu,\ast}\circ \pi_{(a,b)}(\Omega,\mu,\ast)\big\|,
\end{multline}
and by the triangle inequality this is bounded by 
\[\big\|\hat{G}_{a,\gamma}^{\Omega_{m_k},\mu_{m_k},\ast}\circ P_{\Omega_{m_k}}-\hat{G}_{a,\gamma}^{\Omega,\mu,\ast}\circ P_\Omega\big\|+ \frac{1}{a}\left\|\pi_{(a,b)}(\Omega_{m_k},\mu_{m_k},\ast)-\pi_{(a,b)}(\Omega,\mu,\ast)\right\|.\]
Since by (i) and (ii) this converges to zero as $k\to \infty$, we obtain (\ref{E:eivalconv}) for $\lambda_+(\Omega,\mu,\ast)$. We can now move on to the next eigenvalue of $-\mathcal{L}^{\Omega,\mu,\ast}_\gamma$ strictly larger than $\lambda_+(\Omega,\mu,\ast)$ and apply the same arguments. Proceeding inductively, we obtain (\ref{E:eivalconv}) for all nonzero eigenvalues.

To see (iv) for eigenfunctions corresponding to a nonzero eigenvalue $\lambda_n(\Omega,\mu,\ast)$, let $(a,b)$ be a small interval as above containing this but no other eigenvalue of $-\mathcal{L}^{\Omega,\mu,\ast}_\gamma$. Statement (ii) implies that for any corresponding normalized eigenfunction $\varphi_n(\Omega,\mu,\ast)$ we have 
\[\lim_{k\to \infty} \pi_{(a,b)}(\Omega_{m_k},\mu_{m_k},\ast)\varphi_n(\Omega,\mu,\ast)=\pi_{(a,b)}(\Omega,\mu,\ast)\varphi_n(\Omega,\mu,\ast)=\varphi_n(\Omega,\mu,\ast)\quad \text{in $L^2(D)$}\]
and in particular, $\pi_{(a,b)}(\Omega_{m_k},\mu_{m_k},\ast)\varphi_n(\Omega,\mu,\ast)\neq 0$ for large enough $k$. The functions 
\[\varphi^{(m_k)}_n:=\frac{\pi_{(a,b)}(\Omega_{m_k},\mu_{m_k},\ast)\varphi_n(\Omega,\mu,\ast)}{\left\|\pi_{(a,b)}(\Omega_{m_k},\mu_{m_k},\ast)\varphi_n(\Omega,\mu,\ast)\right\|_{L^2(D)}}\]
converge to $\varphi_n(\Omega,\mu,\ast)$ in $L^2(D)$ as $k\to \infty$. Since, again by \cite[Chapter IV, Section 3.4, Theorem 3.16 and Section 3.5]{Kato-1980},
the images of the $\pi_{(a,b)}(\Omega_{m_k},\mu_{m_k},\ast)$ are eigenspaces of the $-\mathcal{L}^{\Omega_{m_k},\mu_{m_k},\ast}_\gamma$ for a single eigenvalue $\lambda_n(\Omega_{m_k},\mu_{m_k},\ast)$, respectively, the functions $\varphi^{(m_k)}_n$ are normalized eigenfunctions for these eigenvalues. By the spectral theorem, (ii) and the continuity of $\Phi$ we also have 
\begin{multline}
\lim_{k\to\infty} \Phi(\mathcal{L}^{\Omega_{m_k},\mu_{m_k},\ast}_\gamma)\varphi^{(m_k)}_n=\lim_{k\to\infty} \Phi(\lambda_n(\Omega_{m_k},\mu_{m_k},\ast))\varphi^{(m_k)}_n\notag\\
=\lim_{k\to\infty} \Phi(\lambda_n(\Omega,\mu,\ast))\varphi_n(\Omega,\mu,\ast)=\lim_{k\to \infty}\Phi(\mathcal{L}^{\Omega,\mu,\ast}_\gamma)\varphi_n(\Omega,\mu,\ast)
\end{multline}
in $L^2(D)$. If all $\lambda_n(\Omega_{m_k},\mu_{m_k},\ast)$ are simple, then $\varphi^{(m_k)}_n=\pm 
\varphi_n(\Omega_{m_k},\mu_{m_k},\ast)$ by normalization. This shows the first part in (iv). To see the second, let $(\varphi^{(m_k)}_n)_k$ now denote a 
sequence of normalized eigenfunctions $\varphi^{(m_k)}_n$ corresponding to $\lambda_n(\Omega_{m_k},\mu_{m_k},\ast)$, respectively.
By (ii), the triangle inequality and the reverse triangle inequality we have $\lim_{k\to \infty} \big\| \pi_{(a,b)}(\Omega,\mu,\ast)\varphi^{(m_k)}_n\big\|=1$,
so that for large enough $k$ we can define 
\[\varphi_{n,k}:=\frac{\pi_{(a,b)}(\Omega,\mu,\ast)\varphi^{(m_k)}_n}{\big\|\pi_{(a,b)}(\Omega,\mu,\ast)\varphi^{(m_k)}_n\big\|_{L^2(D)}}.\]
Again using (ii) we can conclude that
\begin{align}
\lim_{k\to \infty} \big\|\varphi_{n,k}-\varphi^{(m_k)}_n\big\|_{L^2(D)}&\leq \lim_{k\to \infty} \left|\big\|\pi_{(a,b)}(\Omega,\mu,\ast)\varphi^{(m_k)}_n\big\|_{L^2(D)}^{-1}-1\right|\:\big\|\pi_{(a,b)}(\Omega,\mu,\ast)\varphi^{(m_k)}_n\big\|_{L^2(D)}\notag\\
&\qquad \qquad +\lim_{k\to \infty}\big\|\pi_{(a,b)}(\Omega,\mu,\ast)\varphi^{(m_k)}_n-\pi_{(a,b)}(\Omega_{m_k},\mu_{m_k},\ast)\varphi^{(m_k)}_n\big\|_{L^2(D)}\notag\\
&=0.\notag
\end{align}
Since the eigenspace of $\lambda_n(\Omega,\mu,\ast)$ is finite dimensional, the sequence $(\varphi_{n,k})_k$ has a subsequence converging in $L^2(D)$ to a limit $\varphi_n$ which is a normalized eigenfunction of $\mathcal{L}^{\Omega,\mu,\ast}_\gamma$ corresponding to 
$\lambda_n(\Omega,\mu,\ast)$. If the eigenspace is one-dimensional, this limit $\varphi_n$ must be $\pm \varphi_n(\Omega,\mu,\ast)$.

\end{proof}

\section{Shape admissible domains and compactness}\label{S:compactness}

We define classes of admissible domains and prove their compactness. As before we assume $n\geq 2$. The following definition is a generalization of \cite[Definition 2]{HINZ-2020}.

\begin{definition}\label{DefShapeAdmis}
Let $D_0\subset D \subset \mathbb{R}^n$ be non-empty  bounded Lipschitz domains. A triple $(\Omega,\nu,\mu)$ is called \emph{shape admissible} with parameters $D$, $D_0$, $\varepsilon>0$, $n-1\leq s<n$, $\bar c\geq \bar c_{s}>0$, $0\leq d\leq n$ and $ c_d>0$, if 
\begin{enumerate}
\item[(i)] $\Omega$ is an $(\varepsilon,\infty)$-domain such that $D_0\subset  \Omega\subset D$, 
\item[(ii)] $\nu$ is a finite Borel measure $\nu$
with $\supp \nu=\partial \Omega$ satisfying the  (weak, local) lower regularity condition
\begin{equation}\label{E:lowreg-w}
\nu(\overline{B(x,r)})\geq \bar c_{s}\:r^s, \quad x\in  \partial \Omega,\quad  0<r\leq 1,
\end{equation}
and the total mass bound 
\begin{equation}\label{E:totalmass}
\nu(\partial\Omega)\leq \bar c.
\end{equation}
\item[(iii)] $\mu$ is a nonzero Borel measure satisfying (\ref{E:upperdreg}) and $\Gamma=\supp\mu\subset \overline{\Omega}$.
\end{enumerate}
The set of such triples is denoted by $U_{ad}(D,D_0,\varepsilon, s, \bar c_{s}, \bar c, d, c_d)$. We refer to the measures $\nu$ as in (ii) as \emph{boundary volumes} and the the measures $\mu$ in (iii) as \emph{trace volumes}.
\end{definition}

Condition (\ref{E:lowreg-w}) implies that the Hausdorff dimension of $\partial\Omega$ is less or equal $s$, \cite{FALCONER,MATTILA}. We state the corresponding compactness result which had basically been shown in \cite[Theorems 3 and Remark 6]{HINZ-2020}.

\begin{theorem}\label{T:compact} Suppose that the parameters $D$,$D_0$, $\varepsilon$, $s$, $\bar c_{s}$, $\bar c, d, c_d$ are as in Definition~\ref{DefShapeAdmis}. 
\begin{enumerate}
\item[(i)] The class $U_{ad}(D,D_0,\varepsilon, s, \bar c_{s}, \bar c, d, c_d)$ of shape admissible triples is  compact   in the Hausdorff sense, in the sense of characteristic functions, in the sense of compacts, and in the sense of weak converges of the boundary volumes and the trace volumes.
\item[(ii)] If  for a sequence $((\Omega_{m},\nu_m,\mu_m))_m\subset U_{ad}(D,D_0,\varepsilon, s, \bar c_{s}, \bar c, d, c_d)$  the boundary volumes $\nu_m$
converge weakly,
then the domains $\Omega_{m}$ converge  in the Hausdorff sense, in the sense of characteristic functions, and in the sense of compacts.  
\end{enumerate}
\end{theorem}

\cite[Theorems 3]{HINZ-2020} is recovered as a special case of Theorem \ref{T:compact} in the situation that $\mu_m=\nu_m$ for all $m$. In this special case the uniform bound (\ref{E:totalmass}) on the total boundary volumes can be dropped, because it is automatically satisfied as a consequence of (\ref{E:upperdreg}). Note also that (\ref{E:lowreg-w}) implies that the measures are nonzero.

\begin{proof}
Statement (ii) had been proved in \cite[Theorem 3 (ii)]{HINZ-2020}, statement (i) follows from \cite[Theorem 3 (i)]{HINZ-2020} and Banach-Alaoglu, applied to a (sub-)sequence of trace volumes. 
\end{proof}

\begin{remark}\label{R:boundaryvol}\mbox{}
The weak convergence of boundary volumes serves as a convenient tool: The common scaling exponent $s$ and constant $\bar c_{s}$ guarantee that also the limit domain $\Omega$ will have a boundary $\partial\Omega$ with Hausdorff dimension less or equal $s<n$, so that in particular $\lambda^n(\partial \Omega)=0$. This is used to conclude the convergence in the sense of characteristic functions, see \cite[Theorems 2 and 3]{HINZ-2020}. The weak convergence of measures alone does not preserve simple bounds on the Hausdorff dimension of their supports. A bound similar to (\ref{E:lowreg-w}) was also used in \cite[formula (1.6)]{BucurGiacominiTrebeschi2016}.
\end{remark}

Combining Theorem \ref{T:Mosco} with Theorem \ref{T:compact}, we immediately obtain a compactness result for the energy functionals defined in (\ref{extJ--}).

\begin{theorem}\label{T:Moscocompact}
Let $D_0\subset D \subset \mathbb{R}^n$ be non-empty  bounded Lipschitz domains, $\varepsilon>0$, $n-1\leq s<n$, $\bar c\geq \bar c_{s}>0$, $n-2< d\leq n$ and $ c_d>0$. Let $\alpha\geq 0$ , let $\gamma\in C(\overline{D})$ be nonnegative, and let Assumption \ref{A:alphagamma} be satisfied.

Given a sequence $((\Omega_{m},\nu_m,\mu_m))_m\subset U_{ad}(D,D_0,\varepsilon, s, \bar c_{s}, \bar c, d, c_d)$, there are a sequence $(m_k)_k$ with $m_k\uparrow +\infty$ and an admissible triple $(\Omega,\nu,\mu)\subset U_{ad}(D,D_0,\varepsilon, s, \bar c_{s}, \bar c, d, c_d)$ such that 
\[\lim_{k\to \infty} J(\Omega_{m_k},\mu_{m_k})=J(\Omega,\mu)\]
in the sense of Mosco and in the Gamma-sense. The corresponding resolvent operators, spectral projectors, eigenvalues and eigenfunctions converge as stated in Theorem \ref{T:normresconv}.
\end{theorem}

Theorems \ref{T:compactnesspropvar} and \ref{T:Moscocompact} now imply a compactness result for the unique weak solutions of boundary value problems on varying domains.
\begin{corollary}\label{C:compactness}
Let the hypotheses of Theorem \ref{T:Moscocompact} be in force and let $f\in L^2(D)$. Given a sequence $((\Omega_{m},\nu_m,\mu_m))_m\subset U_{ad}(D,D_0,\varepsilon, s, \bar c_{s}, \bar c, d, c_d)$, there are a sequence $(m_k)_k$ with $m_k\uparrow +\infty$ and an admissible triple $(\Omega,\nu,\mu)\subset U_{ad}(D,D_0,\varepsilon, s, \bar c_{s}, \bar c, d, c_d)$ such that for
the continuations by zero $v_{m_k}$ of the unique weak solutions $u_{m_k}$
of (\ref{E:abstractDirichlet}) or (\ref{E:abstractRobin}) on $\Omega_{m_k}$ with zero boundary values we have 
\[\lim_{k\to \infty} v_{m_k}=v\quad \text{in $L^2(D)$},\]
where $v$ is the continuation by zero of the corresponding unique weak solution $u$ on $\Omega$. There is some $u^\ast\in W^{1,2}(D)$ with $u^\ast|_\Omega=u=v|_\Omega$ such that $\lim_{k\to \infty} E_{\Omega_{m_k}}u_{m_k}=u^\ast$ weakly in $W^{1,2}(D)$ and strongly in $W^{1,2}(\Omega)$. Moreover, we have 
\[\lim_{k\to \infty} J(\Omega_{m_k},\mu_{m_k})(u_{m_k})=J(\Omega,\mu)(u).\] 
\end{corollary}

\section{Existence of optimal shapes}\label{S:opti}

From Corollary \ref{C:compactness} we can conclude the existence of optimal shapes minimizing the energy (\ref{extJ--}) within a given class of shape admissible triples. 
\begin{theorem}\label{T:opti}
Let $D_0\subset D \subset \mathbb{R}^n$ be non-empty  bounded Lipschitz domains, $\varepsilon>0$, $n-1\leq s<n$, $\bar c\geq \bar c_{s}>0$, $n-2< d\leq n$ and $ c_d>0$.  Let $\alpha\geq 0$ , let $\gamma\in C(\overline{D})$ be nonnegative, and let Assumption \ref{A:alphagamma} be satisfied.

There is a shape admissible triple $(\Omega_{opt},\nu_{opt},\mu_{opt})\in U_{ad}(D,D_0,\varepsilon, s, \bar c_{s}, \bar c, d, c_d)$ such that 
\[J(\Omega_{opt},\mu_{opt})(u(\Omega_{opt},\mu_{opt}))=\min_{(\Omega,\nu,\mu)\in U_{ad}(D,D_0,\varepsilon, s, \bar c_{s}, \bar c, d, c_d)} J(\Omega,\mu)(u(\Omega,\mu)), \]
where $u(\Omega,\mu)$ denotes the unique weak solution to (\ref{E:abstractDirichlet}) respectively (\ref{E:abstractRobin}) on 
$(\Omega,\nu,\mu)$ with data $\alpha$, $\gamma$, $f$ and $\varphi\equiv 0$.

Moreover, $(\Omega_{opt}, \nu_{opt}, \mu_{opt})$ is the limit of a minimizing sequence 
\[((\Omega_{m},\nu_m,\mu_m))_m\subset U_{ad}(D,D_0,\varepsilon, s, \bar c_{s}, \bar c, d, c_d)\] 
in the Hausdorff sense, the sense of compacts, the sense of characteristic functions and in the sense of weak convergence of the boundary volumes and the trace volumes. There is some $u^\ast\in W^{1,2}(D)$ with $u^\ast|_\Omega=u(\Omega_{opt},\mu_{opt})$ such that $\lim_{m\to \infty} E_{\Omega_{m}}u_{m}=u^\ast$ weakly in $W^{1,2}(D)$ and strongly in $W^{1,2}(\Omega)$.
\end{theorem}

\begin{proof}
Let $((\Omega_m,\nu_m,\mu_m))_m\subset U_{ad}(D,D_0,\varepsilon, s, \bar c_{s}, \bar c, d, c_d)$ be a minimizing sequence for the nonnegative functional $(\Omega,\nu,\mu)\mapsto J(\Omega,\mu)(u(\Omega,\mu))$. The result follows from an application of Corollary \ref{C:compactness} to $((\Omega_m,\nu_m,\mu_m))_m$.
\end{proof}

We formulate a second version of this result, now for boundary value problems in the more narrow sense. Let $U_{ad}'(D,D_0,\varepsilon, s, \bar c_{s}, \bar c, d, c_d)$ be the collection of all shape admissible triples of form $(\Omega,\mu,\mu)$ in $U_{ad}(D,D_0,\varepsilon, s, \bar c_{s}, \bar c, d, c_d)$, i.e. for which boundary and trace volumes agree and in particular, 
$\supp\mu=\partial \Omega$. This is the class defined in \cite[Definition 2]{HINZ-2020}, just with different notation.

\begin{theorem}\label{T:optimu}
Let $D_0\subset D \subset \mathbb{R}^n$ be non-empty  bounded Lipschitz domains, $\varepsilon>0$, $\bar c\geq \bar c_{s}>0$,
\begin{equation}\label{E:specialcase}
n-1\leq s < n\quad \text{and}\quad n-2 \leq  d \leq s.
\end{equation}
Let $\alpha\geq 0$ , let $\gamma\in C(\overline{D})$ be nonnegative, and let Assumption \ref{A:alphagamma} be satisfied. 

There is a shape admissible triple $(\Omega_{opt},\mu_{opt},\mu_{opt})\in U_{ad}'(D,D_0,\varepsilon, s, \bar c_{s}, \bar c, d, c_d)$ such that 
\[J(\Omega_{opt},\mu_{opt})(u(\Omega_{opt},\mu_{opt}))=\min_{(\Omega,\mu,\mu)\in U_{ad}'(D,D_0,\varepsilon, s, \bar c_{s}, \bar c, d, c_d)} J(\Omega,\mu)(u(\Omega,\mu)), \]
where $u(\Omega,\mu)$ denotes the unique weak solution to (\ref{E:abstractDirichlet}) respectively (\ref{E:abstractRobin}) on 
$(\Omega,\nu,\mu)$ with data $\alpha$, $\gamma$, $f$ and $\varphi\equiv 0$.

Moreover, $(\Omega_{opt}, \mu_{opt}, \mu_{opt})$ is the limit of a minimizing sequence 
\[((\Omega_{m},\mu_m,\mu_m))_m\subset U_{ad}'(D,D_0,\varepsilon, s, \bar c_{s}, \bar c, d, c_d)\] 
in the Hausdorff sense, the sense of compacts, the sense of characteristic functions and in the sense of weak convergence of the boundary volumes resp. trace volumes. There is some $u^\ast\in W^{1,2}(D)$ with $u^\ast|_\Omega=u(\Omega_{opt},\mu_{opt})$ such that $\lim_{m\to \infty} E_{\Omega_{m}}u_{m}=u^\ast$ weakly in $W^{1,2}(D)$ and strongly in $W^{1,2}(\Omega)$.
\end{theorem}

\begin{remark}\label{R:boundarycase}
Theorem \ref{T:optimu} provides a quite general result for boundary value problems in the more narrow sense: Suppose that we wish to minimize the functional defined by 
\[J(\Omega,\mu)(v)=\int_\Omega|\nabla v|^2\:dx+ \alpha\int_\Omega v^2\:dx+\int_{\partial\Omega}(\mathrm{Tr}_{\Omega,\Gamma}v)^2\:d\mu,\quad v\in W^{1,2}(\Omega),\]
and $J(\Omega,\mu)(v)=+\infty$ if $v\in L^2(D)\setminus W^{1,2}(\Omega)$. The triples in $U_{ad}'(D,D_0,\varepsilon, s, \bar c_{s}, \bar c, d, c_d)$ then are actually pairs $(\Omega,\mu)$ of bounded $(\varepsilon,\infty)$-domains $\Omega$ and Borel measures $\mu$ such that we have 
\[\mu(B(x,r))\leq c_d\:r^d\quad\text{and}\quad \mu(\overline{B(x,r)})\geq \bar c_s\:r^s,\quad x\in \partial\Omega,\quad 0<r\leq 1,\]
with (\ref{E:specialcase}), see \cite[Definition 2]{HINZ-2020}. The class of domains $\Omega$ for which this is possible is much larger than any class of domains with countably rectifiable boundaries and finite $(n-1)$-dimensional Hausdorff measure. The boundary $\partial\Omega$ could potentially have any Hausdorff dimension in $[n-1,n)$.
\end{remark}

\appendix

\newcommand{\etalchar}[1]{$^{#1}$}


\begin{thebibliography}{MNORP21}

\bibitem[AH96]{AH96}
D.R. Adams and L.I. Hedberg.
\newblock {\em Function spaces and potential theory}, volume 314 of {\em
  Grundlehren der Mathematischen Wissenschaften [Fundamental Principles of
  Mathematical Sciences]}.
\newblock Springer-Verlag, Berlin, 1996.

\bibitem[AHM{\etalchar{+}}17]{AHMNT}
J.~Azzam, S.~Hofmann, J.M. Martell, K.~Nystr\"{o}m, and T.~Toro.
\newblock A new characterization of chord-arc domains.
\newblock {\em J. Eur. Math. Soc. (JEMS)}, 19(4):967--981, 2017.

\bibitem[AW03]{ARENDT-WARMA2003}
{W}. {A}rendt and {M}. {W}arma.
\newblock {T}he {L}aplacian with {R}obin boundary conditions on arbitrary
  domains.
\newblock {\em {P}ot. {A}nal.}, 19:341--363, 2003.

\bibitem[BC07]{BoulkhemairChakib2007}
A.~Boulkhemair and A.~Chakib.
\newblock On the uniform {P}oincar\'e inequality.
\newblock {\em Comm. Partial Diff. Eq.}, 32(9):1439--1447, 2007.

\bibitem[BG15]{BucurGiacomini2015}
D.~Bucur and A.~Giacomini.
\newblock Faber-{K}rahn inequalities for the {R}obin-{L}aplacian: a free
  discontinuity approach.
\newblock {\em Arch. Ration. Mech. Anal.}, 218(2):757--824, 2015.

\bibitem[BG16]{BucurGiacomini2016}
D.~Bucur and A.~Giacomini.
\newblock Shape optimization problems with {R}obin conditions on the free
  boundary.
\newblock {\em Ann. Inst. H. Poincar\'e Anal. Non Lin\'eaire}, 33:1539--1568,
  2016.

\bibitem[BGT16]{BucurGiacominiTrebeschi2016}
D.~Bucur, A.~Giacomini, and P.~Trebeschi.
\newblock The {R}obin-{L}aplacian problem on varying domains.
\newblock {\em Calc. Var. PDE}, 55:133, 2016.

\bibitem[BGT19]{BucurGiacominiTrebeschi2019}
D.~Bucur, A.~Giacomini, and P.~Trebeschi.
\newblock Best constant in {P}oincar\'e inequalities with traces: A free
  discontinuity approach.
\newblock {\em Ann. de l'institute H. Poincar\'e C, Anal. non lin\'eaire},
  36(7):1959--1986, 2019.

\bibitem[BH91]{BH91}
N.~Bouleau and F.~Hirsch.
\newblock {\em Dirichlet {F}orms and {A}nalysis on {W}iener {S}pace}, volume~14
  of {\em deGruyter Studies in Math.}
\newblock deGruyter, Berlin, 1991.

\bibitem[Bie09]{Biegert2009}
M.~Biegert.
\newblock {O}n traces of {S}obolev functions on the boundary of extension
  domains.
\newblock {\em Proc. Amer. Math. Soc.}, 137(12):4169--4176, 2009.

\bibitem[{B}ra06]{Braides}
{A}. {B}raides.
\newblock A handbook of {$\Gamma$}-convergence.
\newblock In M.~{C}hipot and P.~{Q}uittner, editors, {\em {H}andbook of
  {D}ifferential {E}quations: {S}tationary {P}artial {D}ifferential
  {E}quations}, volume~3, chapter~2, pages 101--213. Elsevier, Amsterdam, 2006.

\bibitem[Buc05]{BB2005}
G.~Bucur, D.~Buttazzo.
\newblock {\em Variational methods in shape optimization problems}, volume~65
  of {\em Progress in Nonlinear Differential Equations and their Applications}.
\newblock Birkh\"auser, Boston, 2005.

\bibitem[{C}ap10a]{CAPITANELLI-2010}
{R}. {C}apitanelli.
\newblock {A}symptotics for mixed {D}irichlet-{R}obin problems in irregular
  domains.
\newblock {\em {J}. of {M}ath. {A}nal. and {A}ppl.}, 362(2):450--459, {F}eb
  2010.

\bibitem[{C}ap10b]{CAPITANELLI-2010-1}
{R}. {C}apitanelli.
\newblock {R}obin boundary condition on scale irregular fractals.
\newblock {\em {C}omm. {P}ure and {A}ppl. {A}nal.}, 9(5):1221--1234, {M}ay
  2010.

\bibitem[{D}an00]{DANERS-2000}
{D}. {D}aners.
\newblock Robin boundary value problems on arbitrary domains.
\newblock {\em {T}rans. {A}mer. {M}ath. {S}oc.}, 352(09):4207--4237, {S}ep
  2000.

\bibitem[Dan06]{Daners2006}
D.~Daners.
\newblock A {F}aber-{K}rahn inequality for {R}obin problems in any space
  dimension.
\newblock {\em Math. Ann.}, 335(4):767--785, 2006.

\bibitem[DD97]{DancerDaners1997}
E.N. Dancer and D.~Daners.
\newblock Domain perturbation for elliptic equations subject to {R}obin
  boundary conditions.
\newblock {\em J. Differ. Equ.}, 138(1):86--132, 1997.

\bibitem[DHRPT21]{DHRPT-2021}
{A}. {D}ekkers, {M}. {H}inz, {A}. {R}ozanova {P}ierrat, and {A}. {T}eplyaev.
\newblock {O}ptimal absorption of ultrasound by reflective or isolating
  boundary.
\newblock {\em {P}reprint}, 2021.

\bibitem[DL93]{Dutray-Lions-V6}
R.~Dutray and J.-L. Lions.
\newblock {\em Mathematical {A}nalysis and {N}umerical {M}ethods for {S}cience
  and {T}echnology, {V}olume 6, {E}volution {P}roblems II}.
\newblock Springer, Berlin, 1993.

\bibitem[DM93]{DALMASO-1993}
G.~Dal~Maso.
\newblock {\em An {I}ntroduction to {$\Gamma$}-convergence}, volume~8 of {\em
  Progress in Nonlinear Differential Equations and their Applications}.
\newblock Birkh\"auser, Boston, 1993.

\bibitem[DRPT21]{DRPT-2020}
{A}. {D}ekkers, {A}. {R}ozanova {P}ierrat, and {A}. {T}eplyaev.
\newblock {M}ixed boundary valued problem for linear and nonlinear wave
  equations in domains with fractal boundaries.
\newblock {\em To appear in Calculus of Variations and Partial Differential
  Equations, hal-02514311}, 2021.

\bibitem[EHDR15]{Egert2015}
M.~Egert, R.~Haller-Dintelmann, and J.~Rehberg.
\newblock {H}ardy's inequality for functions vanishing on a part of the
  boundary.
\newblock {\em Pot. Anal.}, 43:49--78, 2015.

\bibitem[{E}va94]{EVANS-1994}
{L}.{C}. {E}vans.
\newblock {\em {P}artial {D}ifferential {E}quations}.
\newblock {G}raduate {S}tudies in {M}athematics, 1994.

\bibitem[{F}al90]{FALCONER}
{K}.~{J}. {F}alconer.
\newblock {\em {F}ractal {G}eometry - {M}athematical {F}oundations and
  {A}pplications}.
\newblock John Wiley and Sons, Chichester, 1990.

\bibitem[FOT94]{FOT94}
M.~Fukushima, Y.~Oshima, and M.~Takeda.
\newblock {\em Dirichlet {F}orms and {S}ymmetric {M}arkov {P}rocesses}.
\newblock deGruyter, Berlin, New York, 1994.

\bibitem[Fug71]{Fuglede71}
B.~Fuglede.
\newblock The quasi topology associated with a countably subadditive set
  function.
\newblock {\em Ann. de l'institute Fourier}, 21(1):123--169, 1971.

\bibitem[GR86]{GIRAULT-1986}
{V}. {G}irault and {P}.-{A}. {R}aviart.
\newblock {\em {F}inite {E}lement {M}ethods for the {N}avier-{S}tokes
  {E}quations, {T}heory and {A}lgorithms}.
\newblock {S}pringer, {N}ew {Y}ork, 1986.

\bibitem[HMRP{\etalchar{+}}21]{HMR-PRT2021}
M.~Hinz, F.~Magoul\`es, A.~Rozanova-Pierrat, M.~Rynkovskaya, and A.~Teplyaev.
\newblock On the existence of optimal shapes in architecture.
\newblock {\em Appl. Math. Model.}, 94:676--687, 2021.

\bibitem[HP18]{HENROT-e}
A.~Henrot and M.~Pierre.
\newblock {\em Shape variation and optimization}, volume~28 of {\em EMS Tracts
  in Mathematics}.
\newblock European Mathematical Society (EMS), Z\"{u}rich, 2018.
\newblock English version of the French publication with additions and updates.

\bibitem[HRPT21]{HINZ-2020}
{M}. {H}inz, {A}. {R}ozanova {P}ierrat, and {A}. {T}eplyaev.
\newblock {N}on-{L}ipschitz uniform domain shape optimization in linear
  acoustics.
\newblock {\em SIAM {J}. {C}ontrol {O}ptim.}, 59(2):1007--1032, 2021.

\bibitem[{J}on81]{JONES-1981}
{P}.{W}. {J}ones.
\newblock {Q}uasi onformal mappings and extendability of functions in {S}obolev
  spaces.
\newblock {\em {A}cta {M}athematica}, 147(1):71--88, {D}ec 1981.

\bibitem[{J}on94]{JONSSON-1994}
{A}. {J}onsson.
\newblock {B}esov spaces on closed subsets of $\mathbb{{R}}\sp n$.
\newblock {\em {T}ransactions of the {A}merican {M}athematical {S}ociety},
  341(1):355--370, {J}an 1994.

\bibitem[{J}on09]{JONSSON-2009}
{A}. {J}onsson.
\newblock {B}esov spaces on closed sets by means of atomic decomposition.
\newblock {\em {C}omplex {V}ariables and {E}lliptic {E}quations},
  54(6):585--611, {J}un 2009.

\bibitem[JW84]{JONSSON-1984}
{A}. {J}onsson and {H}. {W}allin.
\newblock {\em {F}unction spaces on subsets of $\mathbb{{R}}^n$}.
\newblock {M}ath. {R}eports 2, {P}art 1, {H}arwood {A}cad. {P}ubl. {L}ondon,
  1984.

\bibitem[Kat95]{Kato-1980}
T.~Kato.
\newblock {\em Perturbation {T}heory for {L}inear {O}perators}.
\newblock Springer, Berlin, 1995.
\newblock {R}eprint of 1980 edition.

\bibitem[{M}at95]{MATTILA}
P.~{M}attila.
\newblock {\em {G}eometry of {S}ets and {M}easures in {E}uclidean {S}paces -
  {F}ractals and {R}ectifiability}.
\newblock Cambridge University Press, Cambridge, 1995.

\bibitem[{M}az85]{MAZ'JA-1985}
{V}. {M}az’ja.
\newblock {\em {S}obolev {S}paces}.
\newblock {S}pringer {S}er. {S}ov. {M}ath., {S}pringer-{V}erlag, {B}erlin,,
  1985.

\bibitem[MNORP21]{MAGOULES-2020}
{F}. {M}agoul{\`e}s, {T}. {P}.~{K}. {N}guyen, {P}. {O}mn{\`e}s, and {A}.
  {R}ozanova {P}ierrat.
\newblock {O}ptimal absorption of acoustical waves by a boundary.
\newblock {\em SIAM J. Control Optim.}, 59(1):561--583, 2021.

\bibitem[Mos94]{MOSCO94}
U.~Mosco.
\newblock Composite media and asymptotic {D}irichlet forms.
\newblock {\em J. Funct. Anal.}, 123:368--421, 1994.

\bibitem[{N}ys94]{NYSTROM-1994}
{K}. {N}ystr{\"o}m.
\newblock {\em {S}moothness properties of solutions to {D}irichlet problems in
  domains with a fractal boundary}.
\newblock {D}octoral {T}hesis, {U}niversity of {U}me{\"a}, {U}me{\"a}, 1994.

\bibitem[{N}ys96]{NYSTROM-1996}
{K}. {N}ystr{\"o}m.
\newblock {I}ntegrability of {G}reen potentials in fractal domains.
\newblock {\em {A}rkiv f{\"o}r {M}atematik}, 34(2):335--381, {O}ct 1996.

\bibitem[Pos12]{Post-2012}
O.~Post.
\newblock {\em Spectral {A}nalysis on {G}raph-like {S}paces}, volume 2039 of
  {\em Lecture {N}otes {M}ath.}
\newblock Springer, Berlin, 2012.

\bibitem[Rog06]{Rogers}
L.G. Rogers.
\newblock Degree-independent {S}obolev extension on locally uniform domains.
\newblock {\em J. Funct. Anal.}, 235(2):619--665, 2006.

\bibitem[RS80]{REED-SIMON-1980}
M.~Reed and B.~Simon.
\newblock {\em Methods of {M}odern {M}athematical {P}hysics, {I}: {F}unctional
  {A}nalysis}.
\newblock Academic Press, San Diego, 1980.

\bibitem[Rui12]{Ruiz2012}
D.~Ruiz.
\newblock On the uniformity of the constant in the {P}oincar\'e inequality.
\newblock {\em Adv. Nonlin. Studies}, 12(4):889--903, 2012.

\bibitem[Shi21]{shibahara2021gromov}
Hyogo Shibahara.
\newblock Gromov-Hausdorff distance with boundary and its application to Gromov
  hyperbolic spaces and uniform spaces.
\newblock {\em arXiv:2108.03626}, 2021.

\bibitem[Tho15]{Thomas2015}
M.~Thomas.
\newblock Uniform {P}oincar\'e-{S}obolev and isoperimetric inequalities for
  classes of domains.
\newblock {\em Disc. Cont. Dyn. Syst.}, 35(6):2741--2761, 2015.

\bibitem[Tri78]{Triebel78}
H.~Triebel.
\newblock {\em {I}nterpolation {T}heory, {F}unction {S}paces, {D}ifferential
  {O}perators}, volume~18 of {\em {N}orth-{H}olland {M}athematical {L}ibrary}.
\newblock North-Holland, Amsterdam, 1978.

\bibitem[{T}ri97]{TRIEBEL-1997}
{H}. {T}riebel.
\newblock {\em {F}ractals and {S}pectra. {R}elated to {F}ourier {A}nalysis and
  {F}unction {S}paces}.
\newblock {B}irkh{\"a}user, 1997.

\bibitem[Tri02]{Triebel2002}
H.~Triebel.
\newblock {F}unction spaces in {L}ipschitz domains and on {L}ipschitz
  manifolds. {C}haracteristic functions as pointwise multipliers.
\newblock {\em Rev. Mat. Complut.}, 15(2):475--524, 2002.

\bibitem[VS15]{VS15}
Alejandro V\'{e}lez-Santiago.
\newblock Global regularity for a class of quasi-linear local and nonlocal
  elliptic equations on extension domains.
\newblock {\em J. Funct. Anal.}, 269(1):1--46, 2015.

\bibitem[{W}al91]{WALLIN-1991}
{H}. {W}allin.
\newblock {T}he trace to the boundary of {S}obolev spaces on a snowflake.
\newblock {\em {M}anuscripta {M}ath}, 73(1):117--125, {D}ec 1991.

\bibitem[Wei84]{Weidmann1984}
J.~Weidmann.
\newblock Stetige {A}bh\"angigkeit der {E}igenwerte und {E}igenfunktionen
  elliptischer {D}ifferentialoperatoren vom {G}ebiet.
\newblock {\em Math. Scand.}, 54:51--69, 1984.

\bibitem[Wei97]{Weidmann97}
J.~Weidmann.
\newblock Strong operator convergence and spectral theory of ordinary
  differential operators.
\newblock {\em Univ. Iagellonicae Acta Math.}, 34:153--163, 1997.

\bibitem[Zei90]{ZEIDLER}
E.~Zeidler.
\newblock {\em {N}onlinear {F}unctional {A}nalysis and its {A}pplications
  {II}/{A}: {L}inear {M}onotone {O}perators}.
\newblock Springer, New York, 1990.

\end{thebibliography}
\end{document}